\documentclass[10pt]{amsart}

\usepackage{amsmath, amsthm}
\usepackage{eucal}

\usepackage{amssymb}
\usepackage{amscd}
\usepackage{palatino}

\usepackage{latexsym}
\usepackage{epsfig}
\usepackage{graphicx}
\usepackage{amsfonts}
\usepackage{psfrag}

\usepackage{url}
\usepackage{cite}

\usepackage[margin=1.2in]{geometry}

\input xy
\usepackage[PostScript=dvips,noPS]{diagrams}
\xyoption{all}
\UseComputerModernTips

\numberwithin{equation}{section}
\numberwithin{figure}{section}

\newtheorem{theorem}{Theorem}[section]
\newtheorem{lemma}[theorem]{Lemma}
\newtheorem{proposition}[theorem]{Proposition}
\newtheorem{corollary}[theorem]{Corollary}

\newtheorem{remark}[theorem]{Remark}

\theoremstyle{definition}
\newtheorem{definition}[theorem]{Definition}

\newcommand{\C}{{\mathbb{C}}}

\newcommand{\Z}{{\mathbb{Z}}}

\newcommand{\PP}{{\mathbb{P}}}

\newcommand{\into}{\hookrightarrow}

\newcommand{\calH}{{\mathcal H}}

\newcommand{\calK}{{\mathcal K}}

\newcommand{\calM}{{\mathcal M}}

\newcommand{\CC}{{\mathbb C}}
\newcommand{\RR}{{\mathbb R}}
\newcommand{\ZZ}{{\mathbb Z}}
\newcommand{\poly}{{\rm poly}}
\newcommand{\psm}{{\rm psm}}

\newcommand{\sandwich}{{\rm bdd}}

\newcommand{\sgrz}{S \Gr^z_{\apsm}(\calK)}
\newcommand{\sgrzr}{S \Gr^z_{\apsm,r}(\calK)}

\newcommand{\sgrzeqr}{S \Gr^z_{\apsm,=r}(\calK)}
\newcommand{\sgrzrswprime}{S{\Gr'}^z_{\sandwich,r}(\calK)}
\newcommand{\tildeUlambda}{\tilde{U}_\lambda}
\newcommand{\tildeUr}{\tilde{U}_r^G}
\newcommand{\tildeSigmalambda}{\tilde{\Sigma}_\lambda}
\newcommand{\tildeSigmar}{\tilde{\Sigma}_r^G}

\newcommand{\sgrzsw}{S \Gr^z_{\sandwich}(\calK)}
\newcommand{\sgrzrsw}{S \Gr^z_{\sandwich,r}(\calK)}

\newcommand{\apsm}{\alpha({\rm psm})}

\newcommand{\calO}{{\mathcal O}}

\newcommand{\Susp}{S}
\newcommand{\even}{{\rm even}}

\DeclareMathOperator{\SL}{SL}
\DeclareMathOperator{\GL}{GL}

\DeclareMathOperator{\Hom}{Hom}

\DeclareMathOperator{\Ker}{Ker}
\DeclareMathOperator{\CoKer}{Coker}
\DeclareMathOperator{\Gr}{Gr}

\DeclareMathOperator{\diag}{diag}

\DeclareMathOperator{\pt}{pt}
\DeclareMathOperator{\Thom}{Thom}

\DeclareMathOperator{\Ind}{Ind}

\DeclareMathOperator{\quott}{\mathord{/}\!\mathord{/}}

\newcommand{\hsm}{{\hspace{1mm}}}

\newarrow{Equalto}{=}{=}{=}{=}{=}

\newcommand{\basept}{\mathord{*}}

\begin{document}

\title{The module structure of the equivariant $K$-theory of the based loop group of $SU(2)$}

\author{Megumi Harada}
\address{Department of Mathematics and
Statistics\\ McMaster University\\ 1280 Main Street West\\ Hamilton, Ontario L8S4K1\\ Canada}
\email{Megumi.Harada@math.mcmaster.ca}
\urladdr{\url{http://www.math.mcmaster.ca/Megumi.Harada/}}
\thanks{All authors are partially supported by NSERC Discovery Grants. 
The first author is additionally partially supported by
an NSERC University Faculty Award, and an Ontario Ministry of Research
and Innovation Early Researcher Award.}

\author{Lisa C. Jeffrey}
\address{Department of Mathematics \\
University of Toronto \\ Toronto, Ontario \\ Canada}
\email{jeffrey@math.toronto.edu}
\urladdr{\url{http://www.math.toronto.edu/~jeffrey}} 

\author{Paul Selick} 
\address{Department of Mathematics\\ 
University of Toronto\\ Toronto, Ontario \\ Canada} 
\email{selick@math.toronto.edu} 
\urladdr{\url{http://www.math.toronto.edu/~selick}}


\date{\today}


\begin{abstract}
Let $G=SU(2)$ and let $\Omega G$ denote the  space of based
loops in $SU(2)$.  
We explicitly compute the $R(G)$-module structure of the topological equivariant $K$-theory
$K_G^*(\Omega G)$ and in particular show that it is a direct product
of copies of $K^*_G(\pt) \cong R(G)$. (We describe in detail
the $R(G)$-algebra (i.e. product) structure of $K^*_G(\Omega G)$ in a
companion paper.) 
Our proof uses the geometric methods for analyzing loop spaces
introduced by Pressley and Segal (and further developed by
Mitchell). However, Pressley and Segal do not explicitly compute
equivariant $K$-theory and we also need further analysis of the spaces
involved since we work in the equivariant setting. 
With this in mind, we have taken this 
opportunity to expand on the original exposition of Pressley-Segal in
the hope that in doing so, both our results and theirs would be made
accessible to a wider audience. 
\end{abstract}

\maketitle

\setcounter{tocdepth}{1}
\tableofcontents

\section{Introduction}\label{sec:intro}

Let $G$ be a compact connected Lie group. The $G$-equivariant
topological $K$-theory~$K_G^*(X)$ of a topological $G$-space~$X$ is an object of
intrinsic interest, carrying information about~$X$ which reflects the
$G$-action on~$X$. The space $G$ itself, with $G$ acting by
conjugation, and its space of (continuous) based loops~$\Omega G$ with the
induced (pointwise) action, are two examples of natural and important $G$-spaces.
For Lie groups~$G$, the ordinary and Borel-equivariant cohomology
rings $H^*(G)$, $H^*(\Omega G)$, $H_G^*(G)$, and $H_G^*(\Omega G)$
were computed decades ago (with contributions from many people), and
these results are by now considered classical; the same is true of the
computations of the ordinary $K$-theory rings $K^*(G)$ and~$K^*(\Omega
G)$. A brief account of some of these `classical' results is contained in Section~\ref{sec:history}. 
However, computing the \emph{equivariant} $K$-theory of these spaces proved
to be more difficult. For instance, $K^*_G(G)$ was only recently
computed by Brylinski and Zhang in 2000~\cite{BryZha00}.

The chief contribution of this manuscript is a concrete computation of
the module structure of $K^*_G(\Omega G)$ for
the specific case $G =SU(2)$. In addition to being of basic interest, 
this computation is also motivated by
questions from symplectic geometry which we briefly describe at the
end of this introduction. 
For now we note that, at the beginning of work on this manuscript, our
goal was a full and explicit computation of both the module and
product structures of 
$K^*_G (\Omega G)$ when $G = SU(2)$. The present manuscript describes
the module structure, while 
the product structure is computed in the companion paper
(\cite{HarJefSel12b}).

We now proceed to briefly describe our results and methods. 
We view $K^*_G(\Omega G)$ as a
module over $K^*_G = K^*_G(\pt) \cong R(G)$.  
For the following let $\Omega_{\poly}G$ denote the subspace of \emph{polynomial loops}
in $G$, and $\Omega_{\poly, r}G$ the subspace of polynomial loops of
degree~$\leq r$. For a precise definition we refer the reader to equation~\ref{polydef}.
(The spaces $\Omega_{\poly,r}G$ form a filtration of
$\Omega_{\poly}G$.)
With this notation in place we may state the main theorem of this
manuscript (Theorem~\ref{theorem:module structure}), which asserts in
particular that $K^*_G(\Omega G)$ (respectively $K^*_T(\Omega G)$) is
an inverse limit of free $R(G)$-modules (respectively
$R(T)$-modules). 
\begin{theorem}\label{theorem:intro}
Let $G=SU(2)$ and let $T$ denote its maximal torus. Let $\Omega G$ denote the space of based loops in $G$,
equipped with the pointwise conjugation action of $G$. The
$R(G)$-module (respectively $R(T)$-module) $K^*_G(\Omega G)$
(respectively $K^*_T(\Omega G)$) can be described as follows: 
\begin{align*}
K^q_G(\Omega G)&\cong
K^q_G(\Omega_{\poly} G)
\cong \varprojlim\, K^q_G(\Omega_{\poly,r}G) 
\cong \begin{cases}\prod_{r=0}^{\infty} R(G)&\text{if $q$ is even,}\cr
0&\text{if $q$ is odd}\cr
\end{cases}\cr
K^q_T(\Omega G)&\cong
K^q_T(\Omega_{\poly} G)
\cong \varprojlim\, K^q_T(\Omega_{\poly,r}G) 
\cong \begin{cases}\prod_{r=0}^{\infty} R(T)&\text{if $q$ is even,}\cr
0&\text{if $q$ is odd}\cr
\end{cases}\cr
\end{align*}
\qed
\end{theorem}

This theorem should not be surprising to experts for two reasons. 
Firstly, the computation for the
non-equivariant case, using an analogous filtration, follows from the
work of many other authors: for instance, 
James \cite{James1955} described in 1955 a filtration of
spaces of the form $\Omega \Sigma X$, which applies to our situation
of $G=SU(2)$ since $SU(2) \cong S^3 \cong \Sigma S^2$, while 
Pressley and Segal develop a theory for general loop groups $\Omega G$
in \cite{PS86}, which was  further developed by Mitchell in \cite{Mit86}.
Indeed, the technical geometric tools for our 
argument are $G$-equivariant analogues of the ideas of Pressley and
Segal. However, our geometric results are not immediate corollaries of
those in
\cite{PS86}, mainly due to Theorem~\ref{Ur_bundle}.  The
non-equivariant analogue of Theorem~\ref{Ur_bundle} in~\cite{PS86} is
a description of a certain space as a product of contractible spaces
\cite[(8.4.4)]{PS86}, while in Theorem~\ref{Ur_bundle}, we instead get a
non-trivial bundle over $\PP^1$. This geometric distinction is
relevant in our analysis. 
Secondly, 
statements similar to Theorem~\ref{theorem:intro} for the $(T \times
S^1)$-equivariant $K$-theory 
$K^*_{T \times S^1}(\Omega G)$ can be deduced by Kac-Moody methods
(see Kostant-Kumar~\cite{KK90}) or
GKM methods (see e.g. 
Harada-Henriques-Holm~\cite{HHH05}). However, our
$G$-equivariant result is not an immediate corollary of these
torus-equivariant results since, 
for instance, 
it is not always the case for a $G$-space $X$ that
$$K^*_G(X) \cong K^*_T(X)^W$$
 (cf. for example \cite[Example
4.8]{HarLanSja09}), where $W$ is the Weyl group. For this reason we worked instead with 
$G$-equivariant analogues of the approach in \cite{PS86}. 
(In fact, as it turns out,  direct computation confirms that 
the isomorphism $K^*_G(\Omega G) \cong K^*_T(\Omega G)^W$ is
satisfied in our
case.)

We now summarize the strategy of our computation in some more detail.
Let $\Omega_{\poly}G$ and $\Omega_{\psm}G$ denote the subspaces of
polynomial and piecewise smooth loops, respectively, in $\Omega
G$. (Both are defined more precisely below.) One of our key steps is
to prove that there are $G$-equivariant homotopy equivalences
\begin{equation}\label{eq:homotopy equivalent}
\Omega_{\poly} G \simeq_G \Omega_{\psm} G\simeq_G \Omega G.
\end{equation}
This reduces our computation to that of $K^*_G(\Omega_{\poly}G)$. Our
second essential strategy is to analyze the $G$-filtration of
$\Omega_{\poly}G$ by the spaces $\Omega_{\poly, r}G$ for $r \in \Z_{>
  0}$, consisting of loops of polynomial degree $\leq r$. More
specifically, we prove that the filtration quotients
\begin{equation}\label{eq:filtration quotient}
\Omega_{\poly,r} G/\Omega_{\poly,r-1} G
\end{equation}
are $G$-homeomorphic to Thom spaces of complex $G$-vector bundles
over~$\PP^1$, implying that their equivariant $K$-theory can be computed
via the 
(equivariant) Thom isomorphism theorem. 
From this, a computation of $K_G(\Omega_{\poly} G)$ is obtained by induction
and taking the inverse limit.
In order to achieve the results mentioned above, we introduce and use
(following \cite{PS86}) the 
Grassmannian $\Gr^z(\calK)$, where $\calK$ is a separable
Hilbert space described precisely in \S~3. One of the reasons the space
$\Gr^z(\calK)$ is useful is because there is a subspace
$\Gr^z_{\sandwich, r}(\calK) \subset \Gr^z(\calK)$ which is 
$G$-equivariantly homeomorphic to $\Omega_{\poly,r}G$. Thus our proofs
proceed by analyzing appropriate subspaces of $\Gr^z(\calK)$, instead
of working directly with $\Omega G$. 

As already mentioned, the broad outline of our analysis follows the
well-known work of Pressley and Segal \cite{PS86}. As such, in the
current manuscript we have taken this opportunity to significantly
expand on the exposition in \cite{PS86}; in doing so, we hope that
both our results and those in \cite{PS86} will be made accessible to
a broader audience.

\bigskip

We now take a moment to briefly describe two separate motivations for
the current manuscript. The first and primary motivation for the
authors is the symplectic-geometric
context from which this manuscript initially arose. The second
concerns a re-interpretation of the work of Pressley and Segal
\cite{PS86}. 
While none of this exposition
is necessary for understanding the mathematical content of this
manuscript, we hope that this account provides the reader with
additional motivation. 

Let $G$ be a compact Lie group. Let $LG$ denote
the unbased loop group of~$G$. Hamiltonian $LG$-spaces $\mathcal{M}$ and their symplectic quotients $\mathcal{M}{\quott}LG$ 
arise in numerous contexts; a well-known example of a space arising as
such a quotient is the moduli space of flat connections on the trivial
principal $G$-bundle over a compact, connected $2$-manifold $\Sigma$
with boundary $\partial \Sigma = S^1$.  
Quasi-Hamiltonian spaces, introduced by Alekseev, Malkin, and Meinrenken
\cite{AMM}, are similar to Hamiltonian $G$-manifolds but with two
significant differences: first, the $2$-form on the manifold is
neither closed nor nondegenerate, but one has control over the 
image of the differential of the 2-form, as well as the 
degeneracy locus of the 2-form, and second,  
the moment map takes values in the Lie group $G$ rather than 
its Lie algebra. 
Alekseev, Malkin, and Meinrenken show that 
for any Hamiltonian $LG$-space $\calM$, the 
quotient $\calM/\Omega G$ is a compact quasi-Hamiltonian 
$G$-space (note that the based loop group $\Omega G$ is a subgroup of $LG$). Furthermore, they show that every
compact quasi-Hamiltonian $G$-space arises as the quotient of 
a Hamiltonian $LG$-space by the action of $\Omega G$.


An interesting intrinsic problem in the study of Hamiltonian
$LG$-spaces $\calM$ is to find general methods for computing its
equivariant topological invariants, such as $K^*_G(\calM)$ (where $G$
is viewed as the subgroup of constant loops in $LG$). 
Based on the above, one approach to this problem is to consider 
the $\Omega G$-bundle 
\begin{equation}\label{eq:omegaG fibration}
 \Omega G \to \calM \to \calM/\Omega G. 
\end{equation}
To understand
$K^*_G(\calM)$ using this fibration, it would be useful to know $K^*_G$ of the
base $\calM/\Omega G$ (which, by \cite{AMM}, is a compact quasi-Hamiltonian $G$-space and a
symplectic quotient of a Hamiltonian $LG$-space by the action of
$\Omega G$)
and of the fibre $\Omega G$.  
In related work, Bott, Tolman, and Weitsman studied the cohomology of 
a symplectic quotient of a Hamiltonian $LG$-space
and its relation to the equivariant cohomology of 
the original $LG$-space. Specifically, 
they proved that a Hamiltonian $LG$-space $\calM$ satisfies
the `Kirwan surjectivity' property, i.e., 
the ordinary cohomology of the symplectic 
quotient $\calM /\,/ LG$ is obtained as 
the image of the equivariant cohomology of $\calM$
under the restriction map to the inverse image of the
identity under the moment map. A $K$-theoretic analogue of this result
is contained in work of the first and third authors.

With this in mind, it is natural to try to study the equivariant
$K$-theory of explicit and interesting examples of Hamiltonian
$LG$-spaces (such as those associated to gauge theory, as mentioned
above) through the 
fibration~\ref{eq:omegaG fibration} in explicit examples. A
computation of $K^*_G(\Omega G)$ is evidently essential in such an
analysis; this was the original motivation for this manuscript. 
Finally, we note that since there does not exist a Serre spectral sequence for equivariant
$K$-theory, while the first step in using the fibration~\ref{eq:omegaG
  fibration} is clearly a computation of $K^*_G (\Omega
G)$ this does not immediately lead to a computation of $K^*_G(\calM)$.

Our secondary motivation for this manuscript is to formulate 
the work of Pressley and Segal on loop groups in a basis-free fashion.
By doing so, we can prove $G$-equivariant analogues of some of 
their results.
In their well-known book \cite{PS86}, Pressley and Segal choose an ordered basis for~$\CC^n$ and use this to give
a filtration of $\Omega SU(n)$. Mitchell~\cite{Mit86} examined the filtration
quotients of this filtration.
We denote the spaces in their filtration as~$F_k$.
It is important for our purposes to note that $F_k$ is a $G$-subspace of~$\Omega SU(n)$ only when
$k\cong 0$~mod~$n$.
Consequently, we must consider the subsystem consisting only of those $F_k$ such that
$k\cong 0$~mod~$n$.
Of course, doing so changes the filtration quotients, so in this
manuscript we study these new quotients
and show that they are Thom spaces of complex $G$-bundles related to the
tangent bundle of~$\PP^1$.
We also note that for the purposes of the present paper, all we need is the Thom isomorphism,
for which it would have been sufficient merely to show that the filtration
quotients are bundles (without having to explicitly identify the bundles).
However, our followup paper~\cite{HarJefSel12b} makes heavy use of this
identification.
Moreover, one of the key facts used in Pressley-Segal is that
the subspace of $\Omega G$ which they denote $U_\lambda$ is homeomorphic to a
product of contractible spaces.
However, their space $U_\lambda$ is not a $G$-space, and its $G$-orbit (which we
denote~$U^G$) is not a product, nor is it contractible.
Hence we must use a different argument, and in fact, we show that $U^G$ is 
the total space of a nontrivial $G$-bundle over~$\PP^1$. 

Finally, we also explicitly observe in our manuscript that many of the
proofs in the classical  texts of Milnor \cite{Mil63}
and Atiyah \cite{Atiyah} are $G$-equivariant, although the authors did
not point this out (their books were not written in the 
$G$-equivariant context, although the maps were in fact
$G$-equivariant).

\bigskip

\noindent \textbf{Notation.} We standardize some notation and collect
well-known facts to be
used throughout. 

\begin{itemize}

\item
The Lie group $G$ is always $SU(2)$ unless otherwise noted. 
\item
$T$ is the maximal torus of~$G$ given by 
$\left\{\begin{pmatrix}z&0\cr0&z^{-1}\end{pmatrix}
\mathop{\Big\vert} z\in S^1\right\}$.
\item
$W \cong S_2$ is the Weyl group of~$G$.
\item
$R(T)$ is the representation ring of~$T$ and similarly $R(G)$ is the
representation ring of $R(G)$. 
\item
$K_T(\pt) \cong R(T)$ and $K_G(\pt) \cong R(G)$. 
\item Complex 
projective space $\PP^1$ can be $G$-equivariantly identified with $G/T$, where $G$ acts
on $G/T$ by the usual translation. 
\end{itemize}

\section{Classical Results}\label{sec:history}

In this section we give a very brief account of some of the history
associated to computation of topological invariants of $\Omega G$. 

The first results towards the calculation of $H^*(\Omega G)$ for a Lie
group~$G$ were obtained by Bott, who calculated the Betti numbers
of $H^*(\Omega G;\RR)$ using Morse theory.
For the case $G=SU(2)$, he obtained the result
$$H^q(\Omega SU(2);\RR)
=\begin{cases}\RR;&\mbox{if $q$ is even};\cr
0&\mbox{if $q$ is odd}.\cr\end{cases}
$$
Calculations for $\Omega S^n$ by this method are described in Milnor's
classical book on Morse theory~\cite{Mil63}.
This includes the case~$\Omega SU(2)$, since $SU(2)$ is homeomorphic
to the sphere $S^3$.
As on \cite[page~96]{Mil63}, the Morse theory method yields a $CW$-structure
on~$\Omega S^n$ with one cell in degree~$q$ for each $q$ divisible by~$n-1$.
This gives a calculation of $H^*\bigl(\Omega SU(2)\bigr)$ as a graded group
with integer rather than real coefficients:
$$H^q(\Omega SU(2))
=\begin{cases}\ZZ&\mbox{if $q$ is even};\cr
0&\mbox{if $q$ is odd}.\cr\end{cases}
$$

With the introduction of the Serre spectral sequence in~1950 it became
possible to additionally calculate the ring structure on~$H^*(\Omega S^n)$.
For $n$ even, one obtains the isomorphism $H^*(\Omega
S^n)=\Gamma[x]$. Here $\Gamma[x]$ is a
divided polynomial algebra, 
the ring 
additively generated by elements labelled $\gamma_k(x)$ (having
degree~$k|x|$, where $|x|$ is the degree of $x$)
satisfying the multiplicative relations
\[
\gamma_i(x) \gamma_j(x) = \binom{i+j}{i} \gamma_{i+j}(x).
\]
In the case of $H^*(\Omega S^n)$ one takes the degree $|x|$ of $x$ to
be $n-1$. 
This computation allows one to see that $\Omega SU(2)$ is not homotopy equivalent
to the infinite complex projective space~$\PP^\infty$, although they
have the same cohomology groups.

An alternate approach to the calculation of $H^*\bigl(\Omega SU(2)\bigr)$
takes advantage of the fact that $SU(2)\cong S^3$ is a suspension, allowing
the use of work of Bott-Samelson and of James, which we now recount. 

We first recall the work of Bott and Samelson. 
Suppose that $Y$ is a $CW$-complex of finite type such that 
$H_*(X)$ is torsion-free (i.e. has finitely many cells in each
degree) which is an $H$-space such that $H_*(Y)$ is torsion-free.
The $H$-space multiplication map $Y\times Y\to Y$ together with the
K\"{u}nneth Theorem gives $H_*(Y)$ the structure of a Hopf algebra,
whose multiplication is known as the ``Pontrjagin ring structure''
on~$H_*(Y)$.
The Bott-Samelson Theorem says that if $X$ is a connected $CW$-complex
of finite type then the Hopf algebra $H_*(\Omega\Susp X)$ is isomorphic to the
tensor algebra $T\bigl(\tilde{H}_*(X)\bigr)$.
Specifically, the multiplication is that of a tensor algebra and the
comultiplication comes from the
inclusion $H_*(X)\rInto T\bigl(\tilde{H}_*(X)\bigr)$ where the
comultiplication on $H_*(X)$ is dual to the multiplication on~$H^*(X)$.
For the special case $X=S^2$ we obtain the isomorphism 
$H_*(\Omega S^3) \cong T\bigl(\tilde{H}_*(S^2)\bigr)$.
In this case, since $S^3\cong SU(2)$ is a topological group, it is well-known
that the $H$-space structure on $\Omega S^3$ coming from the loop space
structure (by consecutive concatenation of loops) is homotopic to the structure as a topological group (induced
by pointwise multiplication). 
Since $\tilde{H}_*(S^2)$ has rank~one, the tensor algebra reduces to
a polynomial algebra, and we obtain $H_*(\Omega S^3) \cong \ZZ[x]$ where $|x|=2$
and the comultiplication is determined by $\psi(x)=x\otimes 1 + 1\otimes x$.
Thus $\psi(x)^k=\sum_{i+j=k} {k\choose i}\, x^i\otimes x^j$.
Dualizing, defining $\gamma_k[x]$ by
$\langle \gamma_k[x],x^q\rangle=\delta_{kq}$, we find that
$H^*(\Omega S^3)\cong \Gamma[x]$.

Next we recall the work of James, and specifically, the James
filtration. 
Let $X$ be a connected pointed $CW$-complex.
Define $J_k(X):=X^k/\mathord{\sim}$, where
$$(x_1,\ldots,x_{j-1},\basept,x_{j+1},\ldots,x_k)\sim
(x_1,\ldots,x_{j-1}, x_{j+1},\basept,\ldots,x_k).$$
The James construction on $X$ is defined by
$J(X):=\displaystyle{\varinjlim_k}\, J_k(X)$, where $J_k(X)\to J_{k+1}(X)$
is given by $(x_1,\ldots, x_k)\mapsto (x_1,\ldots, x_k,\basept)$.
Thus $J(X)$ is the free monoid on~$X$ (where multiplication is given by concatenation)
with the direct limit topology.
James~\cite{James1955} shows that $J(X)\simeq \Omega\Susp X$ as $H$-spaces.
The space $F_{2r}$, whose study dominates the bulk of this paper, is homotopy
equivalent to $J_{4r}(S^2)$. This can be proved using the Whitehead theorem
and the Cellular 
Approximation Theorem.

Note that the James construction comes with an obvious
filtration $F_k\bigl(J(X)\bigr):=J_k(X)$.
It is clear from the definitions that the filtration quotient
$J_k(X)/J_{k-1}(X)$ is homeomorphic to the $k$-fold smash product~$X^{(k)}$.
Applying this in the case $X=S^2$,
the long exact sequences for the pairs $\bigl(J_k(S^2), J_{k-1}(S^2) \bigr) $
 immediately give the additive structure of
the homology of~$H_*(\Omega S^3)$ by induction. Using the fact that the $H$-space
multiplication on~$J(X)$ is induced by the concatenation map
$J_i(X)\times J_j(X)\to J_{i+j}(X)$ allows us to reproduce the Bott-Samelson
result that $H_*(\Omega S^3)\cong \ZZ[x]$.
But in fact it gives more.
Recall that for connected $CW$-complexes $A$ and $B$  there is a homotopy decomposition
$\Susp (A\times B)\simeq \Susp A\vee \Susp B\vee \Susp (A\wedge B)$
(see Proposition 7.7.6 of \cite{Sel97}).
As a corollary, for any connected $CW$ complex $X$ 
the suspension of the $k$-fold smash product $\Susp X^{(k)} $ 
is a homotopy retract of the suspension of the $k$-fold Cartesian product 
$\Susp X^{k} $.
Thus in the commutative diagram
\begin{diagram}
X^k&\rTo& X^{(k)}\cr
\dTo&&\dEqualto\cr
J_k(X)&\rTo& X^{(k)},\cr
\end{diagram}
after suspending, the top line has a homotopy retraction and therefore
so does the bottom.
This yields James' Theorem $\Susp J_k(X)\simeq \vee_{j=1}^k\Susp
X^{(j)}$, and
taking the limit as $k\to\infty$ we obtain 
\begin{equation}\label{suspjames}
\Susp J(X)\simeq \bigvee_{j=1}^\infty\Susp X^{(j)}.
\end{equation}
Since $\tilde{H}_*(A\vee B)\cong \tilde{H}_*(A)\oplus \tilde{H}_*(B)$
and $\tilde{H}_{*+1}(\Susp Y)\cong \tilde{H}_*(Y)$, 
equation~\ref{suspjames} can be regarded as a geometric version of the
additive portion of the Bott-Samelson calculation 
$$H_*(\Omega \Susp X)\cong T\bigl(\tilde{H}_*(X)\bigr)
\rTo_\cong^{\mbox{additively}}
\bigoplus_{j=0}^\infty \bigl(\tilde{H}_*(X)\bigr)^{\otimes j}.$$
The advantage of the geometric version is that it works equally well for
other (co)homology theories such as $K$-theory.
Indeed, the above also yields 
$$\tilde{K}^q\bigl(J_k(S^2)\bigr)
\rTo_\cong^{\mbox{additively}}\prod_{j=0}^k \tilde{K}^q(S^{2j})
=\begin{cases}\ZZ^{k+1}&\mbox{if $q$ is even};\cr0&\mbox{if $q$ is odd}.\cr
\end{cases}
$$
To get the additive structure of
$\tilde{K}^q\bigl(\Omega SU(2)\bigr)=\tilde{K}^q\bigl(J(S^2)\bigr)$
from this, we must take the limit as $j\to\infty$.
We digress for a moment to discuss this process.

The cohomology functors $\tilde{H}^n(X)$ and~$\tilde{K}^n(X)$ are
representable, 
that is, they are given by the homotopy classes of map~$[X,B_n]$ for an
appropriate $H$-group~$B_n$.
In the case of ordinary cohomology, $\tilde{H}^n(X)=[X,K(\ZZ,n)]$, where
$K(\ZZ,n)$ is an Eilenberg-Mac Lane space, and for $K$-theory 
$$\tilde{K}^n(X)=
\begin{cases}[X,BU]&\mbox{if $n$ is even};\cr
[X,U]&\mbox{if $n$ is odd}.\cr\end{cases}$$
Any reduced cohomology theory~$\tilde{Y}^*(~)$ on $CW$-complexes satisfies
$\tilde{Y}^*(\vee_{i=1}^k X_i)\cong\prod_{i=1}^k \tilde{Y}^*(X_i)$
(this can be seen using Mayer-Vietoris) but a representable theory also has the property
$\tilde{Y}^*(\vee_{i=1}^\infty X_i)\cong\prod_{i=1}^\infty \tilde{Y}^*(X_i)$
for infinite wedges.
A cohomology theory with this property is said to satisfy the ``Milnor
Wedge Axiom''.
Let $X_1\subset X_2\subset \ldots\subset X_k\subset\ldots$ be a sequence of
cofibrations and let $X=\cup_{i=0}^\infty X_i$.
For a cohomology theory satisfying the Milnor wedge axiom, Milnor~\cite{Mil62}
showed using the ``infinite
mapping telescope'' together with a Mayer-Vietoris argument that the cohomology of $X$ is given by the ``Milnor exact sequence''
$$0\to\displaystyle{\varprojlim_n}^1\,Y^{q-1}(X_n)\to Y^q(X)\to
\displaystyle{\varprojlim_n}\,Y^q(X_n)\to0$$
where $\displaystyle{\varprojlim_n}^1$ denotes the first derived functor
of the inverse limit functor.
As a special case of the ``Mittag-Leffler Theorem'', if $(A_n)$ is an inverse
system in which $A_{n+1}\to A_n$ is onto for each~$n$, then
$\displaystyle{\varprojlim_n}^1 A_n=0$, leaving us with
$\tilde{Y}^q(X)\cong \displaystyle{\varprojlim_n}\,\tilde{Y}^q(X_n)$ when
this surjectivity condition is satisfied.

Returning to our case of interest, equation~\ref{suspjames} shows that
$\tilde{Y}^*\bigl(J_{n+1}(X)\bigr)\to Y^*\bigl(J_{n}(X)\bigr)$ is always
a split surjection.
For ordinary cohomology, the system $\tilde{H}^{2n}\bigl(J_k(S^2)\bigr)$
stabilizes once $k\ge n$,
that is, it looks like
$$\ldots =\ZZ = \ZZ = \ldots = \ZZ\to 0= 0=\ldots =0$$
and so we obtain
$$\tilde{H}^q\bigl(\Omega SU(2)\bigr)=
\begin{cases}\ZZ&\mbox{if $q$ is even};\cr0&\mbox{if $q$ is odd},\cr
\end{cases}$$
as before.
For $K$-theory, the system for even~$q$ looks like
$$\ldots \to\prod_{j=0}^{k+1}\ZZ \to\prod_{j=0}^k\ZZ\to
\prod_{j=0}^{k-1}\ZZ\to\ldots\to \ZZ$$
and taking the inverse limit gives
$$\tilde{K}^q\bigl(\Omega SU(2)\bigr)=
\begin{cases}\prod_{j=0}^\infty\ZZ&\mbox{if $q$ is even};
\cr0&\mbox{if $q$ is odd}.\cr\end{cases}$$

The preceding method is specific to $SU(2)$ since it takes advantage of the
fact that $SU(2)$ is a suspension.
To calculate $K^*\bigl(SU(n)\bigr)$, one could instead turn to the
Atiyah-Hirzebruch spectral sequence.
First one calculates $H^*\bigl(\Omega SU(n)\bigr)$ using, for example,
the Serre spectral sequence.
The result is
$H^*\bigl(\Omega SU(n)\bigr)\cong\Gamma[x_1,x_2,\ldots x_{n-1}]$
where the degree of $x_j$ is $2j$.
Since $\Omega SU(n)$ is an infinite $CW$-complex, to avoid convergence issues
in the Atiyah-Hirzebruch spectral sequence we write it as the union of its
skeletons.
Let $X_{k}$ be the $2k$-skeleton of $\Omega SU(n)$ as a $CW$-complex.
Then $H^*(X_k)$ is the truncated divided polynomial algebra
$H^*(X_k)\cong H^*\bigl(\Omega SU(n)\bigr)/\mathord{\sim}$ where
monomials of total degree more than~$k$ are equated to~$0$.
In the Atiyah-Hirzebruch spectral sequence
$E_2^{p,q}=H^p\bigl(X_k;K^q(\pt)\bigr)\Rightarrow K^{p+q}(X_k)$, all
the nonzero terms have even coordinates.
Therefore the spectral sequence collapses to give
$K^*(X_k)\cong H^*(X_k)\otimes K^*(\pt)$.
The spectral sequence is multiplicative, so this isomorphism holds
as $K^*(\pt)$-algebras.
To get $K^*\bigl(\Omega SU(n)\bigr)$ we take the limit as $k\to\infty$.
Since $H^*(X_{k+1})\to H^*(X_k)$ is surjective, 
$K^*(X_{k+1}) \to K^*(X_k)$ is also surjective
so the $\varprojlim^1$ term in the Milnor exact sequence disappears and we
get
$$K^*\bigl(\Omega SU(n)\bigr)\cong
\displaystyle{\varprojlim_k}\,\Gamma_k[x_1,\ldots x_{n-1}]$$
where $\Gamma_k[x_1,\ldots x_{n-1}]$ denotes the truncated divided polynomial
algebra $\Gamma[x_1,\ldots x_{n-1}]/\mathord{\sim}$, in which monomials
of total degree more than~$k$ are equated to~$0$.
Explicitly, $K^{\even}\bigl(\Omega SU(n)\bigr)\cong\hat{\Gamma}[x_1,\ldots,x_{n-1}]$
where 
$\hat{\Gamma}[x_1,\ldots,x_{n-1}]$ denotes
the completion of the  divided polynomial algebra
with respect to its augmentation ideal
(as in Atiyah and MacDonald \cite{}). 

Finally, we recall some known results in equivariant cohomology.
Borel  \cite{Borel} 
showed that the equivariant cohomology $H_G^* (X)$ of a $G$-space $X$
is given by $H^* (X_G)$ where $ X_G := (X \times EG)/G $  is 
known as the Borel construction of~$X$.
(Note that the corresponding statement on equivariant $K$-theory does not hold:
$K_G^*(X)$ is not given by~$K(X_G)$.)
There is a fibration $X \to X_G \to BG$. 
In the  Serre spectral sequence for the 
fibration $$\Omega SU(n) \to \bigl ( \Omega SU(n) \times ESU(n)\bigr ) /SU(2)
\to BSU(n),$$ all the nonzero terms are in even degrees so 
the spectral sequence collapses to give 
$$H^*_{SU(n)} \bigl(\Omega SU(n)\bigr) \cong H^* \bigl(\Omega SU(n)\bigr ) \otimes
H^*\bigl (BSU(n)\bigr). $$
In other words, 
 $H^*_{SU(n)}\bigl (\Omega SU(n)\bigr)$
 is a divided polynomial
algebra $\Gamma_R[x_1, \dots, x_n] $ with 
coefficients in $R$, where $R$ is the equivariant
cohomology of a point (for $n=2$ it is a polynomial ring
$\ZZ[t]$ with one generator~$t$ of degree~$4$).

\section{The Grassmannian $\Gr^z(\calK)$ and its subspaces}\label{sec:grassmannian}

In this section we define a separable Hilbert space $\calK$ and its
associated Grassmannian $\Gr^z(\calK)$. 
Our discussion follows \cite{PS86}. The main results are 
Theorems~\ref{theorem:polyloops-Grassmannian-part1} and
~\ref{theorem:polyloops-Grassmannian-part2}, which provide
$G$-equivariant identifications of appropriate subspaces of
$\Gr^z(\calK)$ with the $G$-spaces $$\Omega_{\psm}U(n),
\Omega_{\psm}SU(n), \Omega_{\poly,r}U(n){\mbox \  and \  }
\Omega_{\poly,r}SU(n),$$ respectively. These identifications 
 allow us, in later sections, to use the language of 
Grassmannians in order to prove results about 
$\Omega_{\psm}G$ and $\Omega_{\poly}G$. 
In this section only, our discussion is valid for $U(n)$ and $SU(n)$
for any $n \geq 2$.

First we quickly recall the definitions of the spaces of (based) loops
in question. Let $H$ be any Lie group. 
As is standard, we let $\Omega H$ denote the space of continuous based loops in $H$ with
basepoint the identity in $H$. 
We define the 
\textbf{piecewise smooth (based) loops} from $S^1$ to $H$ to be
\[
\Omega_\psm H :=\{f\in\Omega H\mid f\text{ is piecewise smooth}\}.
\]
Evidently, $\Omega_\psm H \subseteq \Omega H$. 
Now consider the special case $H = U(n)$. Following \cite{PS86}, we also define the space of \textbf{polynomial based loops}
$\Omega_{\poly}U(n)$ as 
the set of maps $S^1 \to
U(n)$ which can be expressed as Laurent polynomials in $z$, where $z$ is 
the parameter on the circle $S^1$.
More precisely, for $r \geq 0$ we define 
\begin{equation} \label{polydef}
\Omega_{\poly,r}U(n) :=\left\{f: S^1 \to U(n)
\ \Bigg|\  f(1) = \mathbf{1}_{n \times n},\ f=\!\sum_{j=-r(n-1)}^ra_jz^j,
\space a_j \in M(n\times n,\C)
\right\},
\end{equation}
where $\mathbf{1}_{n \times n}$ denotes the identity matrix.
Here the $a_j$ are constant $n \times n$ complex matrices, and
$f(z)$ is required to be unitary (in particular invertible) for all
$z \in S^1 $.
An element in $\Omega_{\poly,r}U(n)$ may also be viewed as an
element of $\Omega_{\poly, r'}U(n)$ for any $r' > r$. Via these
natural inclusions we may define
\[
\Omega_{\poly}U(n) := \bigcup_{r=0}^\infty \Omega_{\poly,r} U(n)
\]
We refer to $\Omega_{\poly}U(n)$ as \textbf{the (space
  of) polynomial (based) loops in $U(n)$}. 

Now let $\calH=L^2(S^1)$ and set
$\calK:=\calH\otimes\C^n$. Let $z$ denote the parameter
on $S^1 \subseteq \C$. 
If we normalize the measure on~$S^1$ so that $\mu(S^1)=1$, then
$\{z^\ell\mid\ell\in\Z\}$ forms an orthonormal Hilbert space basis for~$\calH$.
Define $$\calH_+:=\mbox{closed subspace of $\calH$ spanned by }\{z^\ell\mid \ell\ge 0\}$$
and $$\calH_-:=\calH\ominus\calH_+
=\mbox{closed subspace of $\calH$ spanned by }\{z^\ell\mid \ell< 0\}.$$
Let $\calK_+:=\calH_+\otimes\C^n$ and $\calK_-:=\calH_-\otimes\C^n$.
We now define the Grassmannian (also called the affine Grassmannian)
associated to $\calK$ as
\begin{equation}\label{eq:Grassmannian definition}
\Gr^z(\calK):=
\{\mbox{closed subspaces~$W$ of } \calK\mid zW\subset W\}.
\end{equation}
For $r\ge 0$, we define the following important subspaces of
$\Gr^z(\calK)$:  
\begin{equation}\label{eq:bounded Grassmannian z invariant} 
\Gr^z_{\sandwich,r}\left(\calK\right):=\{W \in \Gr^z(\calK)
\mid z^r\calK_+\subset W \subset z^{-r(n-1)}\calK_+\}.
\end{equation}

We also define the subspaces of {\bf bounded weight} by
\begin{equation}\label{eq:bounded weight Grassmannian}
\Gr_{\sandwich}^z(\calK):=
\bigcup_r \Gr_{\sandwich,r}^z (\calK).
\end{equation}

For $f\in\Omega \GL(n)$, let $M_f:\calK\to\calK$ denote the
multiplication operator $\bigl(M_f(h)\bigr)(z):=f(z)h(z)$ where the
right hand side is the usual multiplication 
 of the vector~$h(z)\in\C^n$ by  the
matrix~$f(z)\in\GL(n)$ We denote
by
\[
W_f := \overline{M_f(\calK_+)} \subset\calK
\]
the closure of the image of $\calK_+$ under~$M_f$.
Note that since $f(z$ is continuous and invertible 
for $z$ in  the compact
set~$S^1$, the operator $M_f$ on $\calK$ is both bounded and invertible.
Moreover, since multiplication by $f(z)$ and $z$ commute and because
$\calK_+$ is closed under multiplication by~$z$, we have $W_f\in\Gr^z(\calK)$.
Thus the map 
\begin{equation}\label{eq:def-alpha}
\alpha: \Omega \GL(n) \to \Gr^z(\calK), \quad f \mapsto W_f 
\end{equation}
is well-defined.

The group $SU(n)$ acts on $\Omega \GL(n)$ by pointwise conjugation and
on $\calK := \calH \otimes \C^n$ (and hence also on $\calK_+$) by
acting on the second factor. Since the action of $SU(n)$ and
multiplication by $z$ commute on $\calK$, there is also an induced
$SU(n)$-action on $\Gr^z(\calK)$. We have the following.

\begin{lemma}
The map $\alpha$ in~\eqref{eq:def-alpha} is $SU(n)$-equivariant. 
\end{lemma}

\begin{proof}
For $g\in SU(n)$,
$$\alpha\bigl(g\cdot f(\cdot)\bigr)=M_{g\cdot f(\cdot)}(\calK_+)
=M_{gf(\cdot)g^{-1}}(\calK_+)=\{gf(\cdot)g^{-1}h(\cdot)\mid h\in\calK_+\}.$$
Since $SU(n)$ acts on $\calK_+ = \calH_+ \otimes \C^n$ through its
second factor only, it follows that $g\calK_+=\calK_+$, and 
$h(z)\in\calK_+$ if and only if $gh(z)\in\calK_+$ for any $g \in
SU(n)$. 
Therefore
$$
\{gf(z)g^{-1}h(z)\mid h(z)\in\calK_+\}
=\{gf(z)h(z)\mid h(z)\in\calK_+\}=g M_f(\calK_+)
=g\alpha\bigl(f(z)\bigr)$$
as desired.
\end{proof}

Now let 
\begin{equation}\label{eq:alpha psm definition}
\alpha_{\psm}:= \alpha|_{\Omega_{\psm} U(n)} : \Omega_\psm U(n)
\to \Gr^z(\calK) 
\end{equation}
 denote the restriction of $\alpha$ to
$\Omega_{\psm}U(n)$. 
We also define 
\begin{equation}\label{eq:def Gr alpha psm}
\Gr_{\apsm}^z(\calK) := \alpha\bigl(\Omega_{\psm}U(n)\bigr)\subset \Gr^z(\calK)
\end{equation}
to be the image under $\alpha$ of the piecewise smooth loops in
$U(n)$, i.e., the image of $\alpha_{\psm}$. 

Our next
major goal, recorded in Theorem~\ref{theorem:polyloops-Grassmannian-part1},
is to show that the restriction of $\alpha$ to the piecewise smooth
loops $\Omega_{\psm} U(n)$ is in fact an equivariant homeomorphism onto
its image.  We accomplish this by defining a map $\beta:
\alpha\bigl(\Omega_{\psm}GL(n)\bigr) \to \Omega_{\psm}U(n)$ which we show to be an
equivariant retraction to~$\alpha_{\psm}$. The
construction and argument requires several preliminary steps.

 For a closed subspace $W$ of~$\calK$, let $P_+^W:W\to\calK_+ $ and   $P_-^W:W\to\calK_-$ denote the orthogonal projections onto $\calK_+$ and $\calK_-$ respectively.
Given a bounded linear operator $B\in B(\calK)$,
the inclusions $\calK_{\pm}\to \calK$ and projections $\calK\to\calK_{\pm}$
give a description of~$B$ as a matrix of operators
$$B=\begin{pmatrix}B_{++}&B_{+-}\cr B_{-+}&B_{--}\cr\end{pmatrix}$$
as in \cite[page 80]{PS86}. 
In particular $(M_f)_{++}$ is the composite
$\calK_+\rTo^{(M_f)|_{\calK_+}}W_f\rTo^{P_+^{W_f}}\calK_+$.
Finally, for a Fredholm operator~$F$, let
$\Ind(F) = \dim\Ker(F) - \dim\CoKer(F)$ denote the index of~$F$.
More generally, we define the index of $W$ by
${\Ind}(W): \dim\bigl(\Ker (P_+^W)\bigr) -\dim\bigl(\CoKer (P_+^W)\bigr) $
for a closed subspace $W$ of $\calK$ for which both are finite-dimensional.
In the case of a subspace $W_f$ arising from a function~$f$ whose
associated $(M_f)_{++}$ is Fredholm, we have 
$\Ind (W_f) = \Ind \bigl ((M_f)_{++}\bigr ) $
since $(M_f)|_{\calK_+}: \calK_+ \to W_f$ is an injection.

We may now state and prove the following. 

\begin{lemma}\label{beta:preparation}
Suppose $f\in \Omega_{\psm} \GL(n)$.
Then $(M_f)_{++}$ is a Fredholm operator and its index $\Ind\bigl((M_f)_{++}\bigr)$  
equals $-n$ times the degree of the homotopy class of the
function $z\mapsto \det\bigl(f(z)\bigr)$
in~$\pi_1(\C\setminus\{0\})\cong\Z$.
\end{lemma}

\begin{proof} 
The fact that $(M_f)_{++}$ is a Fredholm operator
is an immediate
corollary of~\cite[Proposition 6.3.1]{PS86}.
Although their statement is for the continuously differentiable case,
in fact their proof does not require more than piecewise continuous differentiability.

Since the integers are discrete, it follows that $\Ind\left((M_f)_{++}\right)$
depends only on the homotopy class of~$\det f(z)$.
Therefore  to verify the formula for the index it suffices to consider the
special case where $f(z)=\begin{pmatrix}z^k&0\cr0&1\end{pmatrix}$ for some
integer~$k$.
If $k\ge0$ we get $\dim\Ker   (M_f)_{++}=0$ and $\dim\CoKer (M_f)_{++}=nk$ and
if $k\le0$ we get $\dim\Ker (M_f)_{++}=nk$ and $\dim\CoKer (M_f)_{++}=0$
and so the formula holds in both cases.
\end{proof}

Motivated by the above lemma, we wish to focus attention on the subset of $\Gr_{\apsm}(\calK)$ with associated index $0$. Specifically, 
we define the {\bf special Grassmannian} by
\begin{equation}\label{eq:definition special Grassmannian}
\sgrz:=\{W\in\Gr^z_{\apsm}(\calK)\mid \Ind(W)=0\}.
\end{equation}
We also define 
\begin{equation}\label{eq:def sp Gr r}
\sgrzr:=\{W\in\sgrz\mid \dim\Ker(P_+^W)=\dim\CoKer(P_+^W)\le r\}
\end{equation} 
and $$\sgrzeqr:=\{W\in\sgrz\mid \dim\Ker(P_+^W)=\dim\CoKer(P_+^W)= r\}.$$ 

Lemma \ref{lemma:special linear} below explains this terminology: namely, if $f$ takes values in
the special linear group $SL(n)$, then its image under $\alpha$ is in
the special Grassmannian.

Similarly we define
\begin{equation}\label{eq:definition SGr sandwich r}
S\Gr^z_{\sandwich, r}(\calK) := \Gr^z_{\sandwich,r}(\calK) \cap S
\Gr^z_{\apsm}(\calK) = \{W \in \Gr^z_{\sandwich,r}(\calK) \mid
\Ind(W) = 0\}.
\end{equation}
This is precisely the subset of the Grassmannian which we can
equivariantly identify with the polynomial loops $\Omega_{\poly,r}G$
of degree $\leq r$ mentioned in the introduction,
cf. Theorem~\ref{theorem:polyloops-Grassmannian-part2} below. 

With the above results in place we now give an explicit construction
of a map $\beta$ which we will show is an equivariant inverse to
$\alpha_{\psm}$. 
Suppose $W\in \alpha\bigl(\Omega_{\psm}GL(n)\bigr)$, so $W = W_f = M_f(\calK_+)$ for some 
$f \in  \Omega_{\psm}GL(n)$. 
By Lemma~\ref{beta:preparation} we know 
that $(M_f)_{++} $ is Fredholm. Using this fact, 
Pressley and Segal show in ~\cite[p.126]{PS86} that
$\dim(W\ominus zW)=n$. 
Choose an ordered orthonormal basis
$$B=\bigl(w_1(z),w_2(z),\ldots, w_n(z)\bigr)$$
for~$W\ominus zW$.
Let $N_B(z)$  be the      $n \times n$ matrix 
whose $j$th column is formed from the
components of~$w_j \in \calH \otimes \C^n$ with respect to the standard
basis~$\{e_1, e_2, \ldots,e_n\}$ for the $\C^n$ in the second factor. 
It is shown in \cite[p.126]{PS86} that $N_B(z)\in U(n)$ for each $z\in S^1$.
Thus $z \mapsto N_B(z) N_B(1)^{-1}$ is a well-defined loop in~$\Omega U(n)$. 
We now define the map $\beta$ as
\begin{equation}\label{eq:def-beta} 
\beta: \alpha\bigl(\Omega_{\psm}GL(n)\bigr) \to \Omega_{\psm}U(n),
\quad 
\beta(W)(z):=N_B(z)N_B(1)^{-1}. 
\end{equation}

We must first prove the following lemma. 
\begin{lemma} 
The map $\beta$ in~\eqref{eq:def-beta} is well-defined. 
\end{lemma} 

\begin{proof} 
We first show that the image of $\beta$, which a priori is an element
of $\Omega U(n)$, in fact lands in $\Omega_{\psm}U(n)$. This follows
from the construction of $N_B$ and the fact that $f$ is by definition
piecewise smooth. 
Next 
note that for a different choice of ordered
orthonormal basis $B'=\bigl(w'_1,w'_2,\ldots w'_n\bigr)$ of~$W\ominus zW$, the
matrix $N_{B'}(z)$ would be related to $N_B$ by $N_{B'}(z)=N_B(z)A$
where $A\in U(n)$ is the (constant) linear transformation taking the ordered
basis $B'$ to~$B$.
Therefore $$N_{B'}(z) N_{B'}(1)^{-1}= N_B(z) A A^{-1} N_B(1)^{-1} =
N_B(z) N_B(1)^{-1}.$$
Hence $\beta$ is well-defined. 
\end{proof}

We next prove that $\beta$ respects the relevant group actions. 

\begin{lemma}
The map $\beta$ in~\eqref{eq:def-beta} is $U(n)$-equivariant.
\end{lemma} 

\begin{proof}
Let $W\in \alpha\bigl(\Omega_{\psm}GL(n)\bigr)$ with choice of ordered basis $B=\bigl(w_1(z),w_2(z),\ldots w_n(z)\bigr)$ for $W \ominus zW$. 
Then $gB(z):=\bigl(gw_1(z),gw_2(z),\ldots, gw_n(z)\bigr)$ is a valid ordered basis for $g\cdot W \ominus z(g\cdot W)$.
Therefore $N_{gB}(z)=g N_B(z)$ and so
$$\beta(g\cdot W)=gN_B(z)N_B(1)^{-1}g^{-1}=g\cdot\beta(W)$$
as desired. 
\end{proof}

We are ready to prove that $\alpha_{\psm}$ is an equivariant
homeomorphism onto its image, with equivariant inverse given by the
above map $\beta$.  
This is an 
equivariant analogue of \cite[Theorem~8.3.2]{PS86}.

\begin{theorem}\label{theorem:polyloops-Grassmannian-part1}
{\ }

The map $\alpha_{\psm}$ is an $SU(n)$-equivariant homeomorphism
from $\Omega_{\psm}U(n)$ to  its image
$\Gr^z_{\apsm}(\calK)$, with (equivariant) inverse given by
$\beta \vert_{\Gr^z_{\apsm}(\calK)}$. 
\end{theorem}

\begin{proof}
For $f\in\Omega_{\psm}U(n)$, the set $B=\{f(z)e_1, ,\ldots, f(z)e_n\}$
forms an orthonormal basis for~$W_f \ominus zW_f$.
Writing the components of these vectors as columns of a matrix simply
reproduces the matrix~$f(z)$.
That is, $N_B(z)=f(z)$ and so $\beta\bigl(\alpha(f)\bigr)=f$.
Now suppose $W\in\Gr^z_{\apsm}(\calK)$. 
The definition of~$\beta$ requires
that the rows of~$\beta(W)$ form an orthonormal basis for~$W \ominus
zW$, which
means that they form a generating set for $W$ as a
$\C[z]$-module. Hence $\alpha\bigl(\beta(W)\bigr)=W$.
Since $\alpha$ and~$\beta$ are continuous by \cite[page 129]{PS86}
and are $SU(n)$-equivariant, it follows
that $\alpha_{\psm}$ is an $SU(n)$-equivariant homeomorphism from
$\Omega_{\psm}U(n)$ to its image.
\end{proof}

\begin{lemma}\label{lemma:special linear}
If $f \in \Omega_{\psm} \SL(n)$ then $W_f\in\sgrz$.
\end{lemma}

\begin{proof}
If $f\in\Omega_{\psm}\GL(n)$ then $W_f=W_{\tilde{f}}$ where
$\tilde{f}:= \beta\circ\alpha(f)\in \Omega_{\psm}U(n)$.
This exhibits $W_f$ as an element of $\Gr^z_{\apsm}(\calK)$.
If $f\in\Omega_{\psm}\SL(n)$ then $z \mapsto \det\bigl(f(z)\bigr)$ is the constant function~$1$. This has degree $0$, so by Lemma~\ref{beta:preparation} we conclude $\Ind\bigl((M_f)_{++}\bigr)=0$ and hence $W_f \in \sgrz$. 
\end{proof}

We next show that $\alpha$ and $\beta$
also behave well with respect to the filtrations on the spaces
$\Omega_{\poly}U(n)$ and $\Gr^z_{\sandwich}(\calK)$.
We first record a simple lemma used in the proof. 

\begin{lemma}\label{complexpoly}
Let $p(z)$ be a polynomial with complex coefficients which has a nonzero constant term and
satisfies $|p(z)|=1$ for all $z\in S^1$.
Then $p(z)$ is a constant.
\end{lemma}
\begin{proof}
Let $m=\deg p$.
Set $f(z):=p(z)\bar{p}(z^{-1})$ where $\bar{p}(w)=\overline{p(\bar{w})}$ is the
polynomial obtained by taking the complex conjugate of each of the coefficients
in~$p(z)$.
Then $f(z)$ is analytic on the complement of~$\{0\}$, with a pole of order~$m$
at~$0$.
On the unit circle we have
$$f(z)=p(z)\bar{p}(z^{-1})=p(z)\bar{p}(\bar{z})=|p(z)|^2 = 1. $$
If two analytic functions agree on a convergent sequence then they are equal
and so $f(z) = 1$ on $\CC \setminus \{0\} $. But then the singularity
of $f(z)$ at the origin is removable, which implies that 
$m = 0 $. 
\end{proof}

\begin{proposition}\label{prop:alpha beta respect filtration}
Let $\alpha$ and $\beta$ be the maps defined in~\eqref{eq:def-alpha} and~\eqref{eq:def-beta} respectively. Then: 
\begin{enumerate}
\item
 $\alpha\bigl(\Omega_{\poly,r}U(n)\bigr)\subset\Gr_{\sandwich,r}^z(\calK)$.  
\item The restriction of $\alpha$ to $\Omega_{\poly,r}U(n)\bigr)$ is a
surjection to $\Gr_{\sandwich,r}^z(\calK)$.
In particular, $\Gr^z_{\sandwich,r}(\calK)$ is a subset of
$\Gr^z_{\apsm}(\calK)$.
Thus $\beta$ is defined on $\Gr^z_{\sandwich,r}(\calK)$ and
$\beta\big(\Gr_{\sandwich,r}^z(\calK)\bigr)\subset\Omega_{\poly,r}U(n)$.
\item
$\beta\big(S\Gr_{\sandwich,r}^z(\calK)\bigr)\subset\Omega_{\poly,r}SU(n)$.
\end{enumerate}
\end{proposition}

\begin{proof}
{\ }
\begin{enumerate}
\item
Let $f\in\Omega_{\poly,r}U(n)$.
Then $z^{r(n-1)}f(z)$ is a polynomial in~$z$ and so if $h(z)\in\calK_+$
then $z^{r(n-1)}f(z)h(z)\in\calK_+$.
It follows that $f(z)h(z)\in z^{-r(n-1)}\calK_+$.
Thus $W_f\subset z^{-r(n-1)}\calK_+$.
Next we show $z^r\calK_+ \subset W_f$. Let $h(z)\in z^r\calK_+$. Then
$z^{-r}h(z)\in\calK_+$.
Since $f(z)\in\Omega_{\poly,r}U(n)$, it is a Laurent polynomial with powers of $z$ ranging between $z^{-r(n-1)}$ and $z^r$. Hence its adjoint~$f(z)^*=f(z)^{-1}$ has powers of $z$ between $z^{-r}$ and $z^{r(n-1)}$. In particular $z^r f(z)^{-1}$ is polynomial (i.e. has no negative powers of $z$).
Hence $$f(z)^{-1}\bigl(h(z)\bigr)=z^rf(z)^{-1}\bigl(z^{-r}h(z)\bigr)$$
lies in~$\calK_+$, which in turn implies 
$h(z)=f(z)\Bigl(f(z)^{-1}\bigl(h(z)\bigr)\Bigr)$ lies
in~$f(z)(\calK_+)\subset W_f$, as desired. 
\item
Given part~(1), the equality
 $\alpha\bigl(\Omega_{\poly,r}U(n)\bigr)=\Gr_{\sandwich,r}^z(\calK)$.  
can be verified by counting dimensions.
Alternatively we can construct the preimage (under $\alpha$) of a subspace
$W\in\Gr_{\sandwich,r}^z(\calK)$ as follows.
The usual Gram-Schmidt procedure in the finite dimensional vector space
$$W/z^r \calK\subset z^{-r(n-1)}\calK/z^r\calK\cong\C^{n^2r}$$
can be used to construct an orthonormal basis for $W\ominus zW$ in which the
components of all elements are Laurent polynomials with nonzero coefficients
of $z^k$ only for $-r(n-1)\le k\le r$. 
(Note that the  normalization portion of this Gram-Schmidt process requires
only divisions by positive real numbers, not polynomials,
so the resulting elements are still (Laurent) polynomials.)
Now apply the construction given in the definition of~$\beta$.
The resulting function will lie in~$\Omega_{\poly,r}U(n)$.
\item
Suppose $W\in S\Gr^z_\sandwich$.
By part~(2), $\beta(W)\in\Omega U(n)$ and we must show that
$$\det\bigl(\beta(W)\bigr)(z)=1$$ for all $z\in S^1$.
Let $h(z)=\det\bigl(\beta(W)\bigr)(z)$.
Using $h(z)\in U(n)$, we know $h(1)=1$ and $|h(z)|=1$ for all $z\in S^1$.
Also $\Ind(W)=0$, and so $h(z)\simeq 1$.
Since $h(z)$ is a Laurent polynomial, there exists $r$ such that
$p(z):=z^rh(z)$ is a polynomial with nonzero constant term.
The polynomial $p(z)$ satisfies $|p(z)|=1$ for all $z\in S^1$, so 
by Lemma~\ref{complexpoly}, $p(z)$ is a constant function.
We can evaluate the constant using $h(1)=1$ to deduce that $p(z)\equiv 1$.
Thus $h(z)=z^{-r}$.
The fact that $h(z)\simeq 1$ tells us that $r=0$, so $h(z)\equiv1$.
\end{enumerate}
\end{proof}

Now let 
\begin{equation}\label{eq:definition alpha poly}
 \alpha_{\poly,r}:=
\alpha|_{\Omega_{\poly,r} U(n)}: \Omega_{\poly,r } U(n) \to
\Gr^z_{\sandwich,r} (\calK)
\end{equation}
 denote the restriction of $\alpha_{\psm}$ to
$\Omega_{\poly,r} U(n).$
We are ready to state and prove the analogues of
Theorem~\ref{theorem:polyloops-Grassmannian-part1} for the relevant
subspaces of~$\Omega_{\psm}U(n)$. 
The first claim of the theorem below is an analogue of \cite[Proposition
8.3.3(i)]{PS86}.

\begin{theorem}\label{theorem:polyloops-Grassmannian-part2}
{\ }
\begin{enumerate}
\item
The map $\alpha_{\poly,r}$ is an $SU(n)$-equivariant homeomorphism
from $\Omega_{\poly, r}U(n)$ to $\Gr_{\sandwich,r}^z\left(\calK\right)$.
\item
The restriction of $\alpha_{\psm}$ to $\Omega_{\psm}SU(n)$ is
an $SU(n)$-equivariant homeomorphism from $\Omega_{\psm}SU(n)$ to its image
in $S\Gr^z_{\apsm}(\calK):=\{W\in\Gr^z_{\apsm}(\calK)\mid\Ind(W)=0\}$.
\item
The restriction of $\alpha_{\poly,r}$ to $\Omega_{\poly,r}SU(n)$ is
an $SU(n)$-equivariant homeomorphism from
$$\Omega_{\poly,r}SU(n)$$
to~$S\Gr_{\sandwich, r}^z(\calK):=\{W\in\Gr_{\sandwich, r}^z(\calK)\mid\Ind(W)=0\}$.
\end{enumerate}
\end{theorem}

\begin{proof}
The first claim follows from
Theorem~\ref{theorem:polyloops-Grassmannian-part1} 
and 
Proposition~\ref{prop:alpha beta respect filtration}, since the
restriction of a homeomorphism to a subspace induces a homeomorphism
onto its image. 
Restricting to the connected component of the identity in $\Omega_{\psm}U(n)$
and $\Omega_{\poly,r}U(n)$ respectively gives the last two claims. 
\end{proof}

\section{Description of filtration quotients as Thom spaces}\label{sec:filtration}
Henceforth we restrict attention to the case $n=2$. In particular, we return to our main case
$G=SU(2)$. 
Our main result in the previous section, Theorem~\ref{theorem:polyloops-Grassmannian-part2}, shows that the spaces $\Omega_{\poly,r}G$, which provide a natural filtration of $\Omega_{\poly}G$, may be (equivariantly) identified with $\sgrzrsw$, so these spaces will be the main focus of our analysis below. For simplicity we introduce the notation
\begin{equation}\label{eq:def F2r}
F_{2r} := \sgrzrsw.
\end{equation}
The main result of this brief section is a concrete geometric description of the quotients $F_{2r}/F_{2r-2}$ as Thom spaces of vector bundles. 
Here and below, $\gamma$ denotes the tautological bundle over $\PP^1$. Also we equip $\C^2$ with the standard hermitian metric, and let $\perp$ denote the orthogonal complement with respect to this metric. The following definition is useful for our description of $F_{2r}/F_{2r-2}$. 

\begin{definition}\label{definition:tau}  
Let $\tau$ denote the $G$-equivariant
complex line bundle over~$\PP^1$ whose total space is
$$
\bigl\{(u,v)\mid u\in S^3\subset\C^2, v\in (u^\perp)\bigr\}
/\mathord{\sim}$$
where the equivalence relation is given by $(u,v)\sim(\zeta u,\zeta v)$ for $\zeta\in S^1$, and with projection to $\PP^1$ given by 
$[(u,v)] \mapsto [u] \in \PP^1$. The $G$-action is defined by $g \cdot [(u,v)] := [(gu, gv)]$. 
\end{definition}

The notation $\tau$ is justified by the following proposition. 

\begin{proposition}\label{prop:tau}
The bundle $\tau$ of Definition~\ref{definition:tau} is
$G$-equivariantly isomorphic to
the tangent bundle of~$\PP^1$.
\end{proposition}

\begin{proof}
The tangent bundle of $\PP^1$ can be identified with 
$\Hom(\gamma, \gamma^\perp)$, as Milnor  shows in the proof of \cite[Theorem~14.10] {Mil74}.
(Although \cite{Mil74} discusses only the non-equivariant case, it is in fact easy to
check that the maps defined there are $G$-equivariant.)
Thus it suffices to show that $\tau$ is $G$-equivariantly isomorphic to $\Hom(\gamma, \gamma^\perp)$. An element $[(u,v)]$ in the total space of $\tau$ 
  uniquely specifies a linear map $\phi_{[(u,v)]}: \gamma \to \gamma^\perp$ by
  setting $\phi_{[(u,v)]}(u) = v$. By linearity, $\phi_{[(u,v)]}(\zeta
  u) = \zeta v$, so this is well-defined on equivalence classes, and it
  is straightforward to see this is a bijective correspondence which
  is equivariant and linear on fibers. 
\end{proof}

We now proceed to the main result of this section, Proposition~\ref{prop:quotient}. 

\begin{lemma}\label{lemma:unique w}
Let $r \in \ZZ$ with $r >0$.
Let $W \in F_{2r} \setminus F_{2r-2}. $
Then there exists $w \in W$ of the form
$$ w = z^{-r}u_0  +z^{-r+1} u_1 
+ \ldots + z^{r-2} u_{2r-2} + z^{r-1} u_{2r-1}$$
with 
$u_j \in \CC^2 $, $u_0 \ne 0 $ and $u_j \perp u_0$ for $j > 0 $.
Moreover, up to 
a nonzero complex scalar multiple, $w$ is uniquely determined by~$W$. 
\end{lemma} 

\begin{proof}
  That there exists such a $w$ follows from the assumption that $W$ is
  in $F_{2r}$ but not in $F_{2r-2}$. The uniqueness of such $w$ up to
  multiplication by a scalar multiple follows from the assumption that
  $W$ is closed under multiplication by $z$ and the fact that
  $\dim_\C(W/z^r \calK_+) = 2r$ (which in turn follows from the
  assumption that $\Ind((M_f)_{++})=0$).
\end{proof}

Our main geometric proposition follows immediately from the preceding discussion. 

\begin{proposition}\label{prop:quotient}  
Let \(r \in \Z\) and \(r \geq 0.\) The quotient space $F_{2r}/F_{2r-2}$
is $G$-equivariantly homeomorphic to  $\Thom(\tau^{2r-1})$.
\qed
\end{proposition}

Using Proposition~\ref{prop:quotient} and the Thom isomorphism  in equivariant $K$-theory (\cite[Theorem~6.1.4]{AtiyahSegal:1965} or
\cite[Theorem~3.1]{Greenlees}) yields the following.

\begin{theorem}\label{KG:F2r}
\[
K_G^q(F_{2r})\cong
\begin{cases}\prod_{k=0}^r R(G)
&\mbox{if $q$ even};
\cr0&\mbox{if $q$ odd}.\end{cases}.
\]
Considering $F_{2r}$ as a $T$-space under the natural restriction of
the $G$-action to its maximal torus $T$, we also have 
\[
K_T^q(F_{2r})\cong
\begin{cases}\prod_{k=0}^r R(T)
&\mbox{if $q$ even};
\cr0&\mbox{if $q$ odd}.\end{cases}
\]
\end{theorem}

\begin{proof}
In the exact sequence
\begin{equation}\label{exseq}
\ldots\to K_G^*(F_{2r},F_{2(r-1)})
\to K_G^*(F_{2r}) \to K_G^*(F_{2(r-1)}
\to\dots
\end{equation}
associated to the pair~$(F_{2r},F_{2r-2})$,
the equivariant Thom isomorphism gives
$$K_G^*(F_{2r},F_{2(r-1)})\cong
\tilde{K}_G^q\bigl(\Thom(\tau^{2r-1})\bigr) \cong K_G^q(\PP^1)\cong
\begin{cases}R(G)\oplus R(G)&\mbox{if $q$ even};
\cr0&\mbox{if $q$ odd}.\end{cases}$$
Thus the exact sequence decomposes into a collection of short exact sequences. These short exact sequences 
split since $K_G^*(F_{2(r-1)})$ is a free $R(G)$-module by induction
(where the base case for the induction is~$F_0=\pt)$.
The first statement then follows. 
The proof of the statement for $K_T^*$ is identical.
\end{proof}

Our final result for $K_G(\Omega_{\poly}G)$ is obtained by taking an inverse
limit.

\begin{theorem}\label{theorem:module for poly loops} 
Let $G=SU(2)$ and let $T$ denote its maximal torus. Let $\Omega_{\poly} G$ denote the space of based polynomial loops in $G$,
equipped with the pointwise conjugation action of $G$. The
$R(G)$-module (respectively $R(T)$-module) $K^*_G(\Omega_{\poly} G)$
(respectively $K^*_T(\Omega_{\poly} G)$) can be described as follows: 
\begin{align*}
K^q_G(\Omega_{\poly} G)
\cong \varprojlim\, K^q_G(\Omega_{\poly,r}G) 
\cong \begin{cases}\prod_{r=0}^{\infty} R(G)&\text{if $q$ is even,}\cr
0&\text{if $q$ is odd};\cr
\end{cases}\cr
K^q_T(\Omega_{\poly} G)
\cong \varprojlim\, K^q_T(\Omega_{\poly,r}G) 
\cong \begin{cases}\prod_{r=0}^{\infty} R(T)&\text{if $q$ is even,}\cr
0&\text{if $q$ is odd}.\cr
\end{cases}\cr
\end{align*}
\end{theorem} 

\begin{proof} Since 
$$\bigcup_r F_{2r} \cong_G \Omega_{\poly}(SU(2))$$
by Theorem~\ref{theorem:polyloops-Grassmannian-part2},
the computation of $K_G\bigl(\Omega_{\poly} SU(2)\bigr)$ as an $R(G)$-module
is obtained by taking the inverse limit of $K^*_G(F_{2r})$.
More specifically, 
the Milnor exact sequence~\cite{Sel97} implies that
$K^*_G(\Omega_{\poly}G) $ is given by $\varprojlim\, K^*_G(\Omega_{\poly,r}G)$
or equivalently by $\varprojlim\, K^*_G(F_{2r})$. The result follows. 
\end{proof}

\section{The $G$-homotopy equivalence $S{\Gr'}_{\sandwich,r}^z(\calK) \to S\Gr_{\apsm,r}^z(\calK)$}\label{sec:sigma G and U G}

The goal of the rest of the manuscript is to prove that the inclusion
$\Omega_{\poly}G \into \Omega G$ is a $G$-homotopy equivalence; we do
this by proving separately that the inclusions  $\Omega_{\poly}G \into \Omega_{\psm}G$
and $\Omega_{\psm}G \into \Omega G$ are both $G$-homotopy equivalences.
This then reduces the computation of $K^*_G(\Omega G)$ to that of
$K^*_G(\Omega_{\poly}G)$, which was recorded in
Theorem~\ref{theorem:module for poly loops} above.  To this end, the
goal of the present section is to show that the inclusion
$S{\Gr'}_{\sandwich,r}^z(\calK) \into \sgrzr$ is a $G$-equivariant
homotopy equivalence for any fixed $r>0$
(Theorem~\ref{theorem:milnorlimit}), where
$S{\Gr'}_{\sandwich,r}^z(\calK)$ is a certain subspace (similar in
spirit to $\sgrzrsw$) to be defined precisely below.
This is the main
technical ingredient in our proof 
that $\Omega_{\poly}G \into \Omega_{\psm}G$
is a $G$-homotopy equivalence (Theorem~\ref{theorem:poly into psm}). 

We begin with an explicit description of a map 
\begin{equation}
  \label{eq:definition pi}
\pi: \sgrzeqr\to \PP^1.  
\end{equation}
The map $\pi$ will play a significant role in
the technical arguments below, where we show 
(Proposition~\ref{prop: pi is a homotopy equivalence})
that $\pi$ is a $G$-homotopy equivalence between a certain subset~$\Sigma_r^G$
of~$\sgrzeqr$ and~$\PP^1$, which in turn allows us to prove
our main geometric result (Theorem~\ref{Ur_bundle}). 

We will formulate the construction of~$\pi$ in an entirely coordinate-free
manner, in particular without choosing either a maximal torus of $G$ or an
ordered basis of $\C^2$.
Suppose
$$W=W_f=\alpha(f)\in S\Gr_{\apsm}^z(\calK)$$
for some $f\in \Omega_{\psm} G$.
By the earlier discussion, $(M_f)_{++}$ is a Fredholm operator with
index~$0$, so we define 
$$r=r(W):=\dim\Ker\bigl((M_f)_{++}\bigr)=\dim\CoKer\bigl((M_f)_{++}\bigr).$$
We sometimes refer to $r(W)$ as the {\em rank} of~$W$.

Given a Laurent series $g(z) =
\sum_k a_k z^k$ with $a_k \in \C^2$ we let $\deg(g)$ denote the
maximum of the set $\{k\mid a_k \neq 0\}$ or $\infty$ if there is no
maximum. Thus $\deg(g) \in \Z \cup \{\infty\}$. Note that if $\deg(g)
< 0$ then there is no `Taylor part' to the Laurent series, i.e. there
are no non-zero terms $a_k z^k$ with $k\geq 0$. Also for any $k \in
\Z$ we denote by
$\langle z^k \rangle \otimes \C^2$ the $2$-dimensional subspace of
$\calH \otimes \C^2$ spanned by the vectors with a $z^k$ coefficient. 

Suppose now $W = W_f$
where $r(W) = r > 0$. We may think of $w\in W$ as
a Laurent series in the variable $z$ with coefficients in
$\C^2$. Consider the set 
\[
S_W := \{w \in W \hsm \vert \hsm \deg(w) = -1\}.
\]
Observe that $S_W \neq \emptyset$ since otherwise
$\Ker\bigl((M_f)_{++}\bigr) = 0$. Let $V$ be the subspace of
$\langle z^{-1}\rangle \otimes \C^2 \cong \C^2$ spanned by the leading
coefficients of the elements of $S_W$. 

\begin{lemma}\label{lemma:dim V is 1}
Let $W = W_f$ with $r(W)>0$, and let $V$ be as above. Then $\dim_\C(V)=1$. 
\end{lemma} 

\begin{proof} 
Since $S_W \neq \emptyset$, we know $\dim_\C(V) > 0$. On the other
hand, if $\dim_\C(V) = 2$ then these elements span $\langle
z^{-1}\rangle \otimes \C^2$ and then the leading coefficients of the
set $z^{k+1}S_W$ would span $\langle z^{k} \rangle \otimes \C^2$ for
all $k \geq 0$, which would in turn imply
$\CoKer\bigl((M_f)_{++}\bigr) = 0$. This is a contradiction since we
assumed $r=r(W) > 0$. The conclusion follows. 
\end{proof} 

 From Lemma~\ref{lemma:dim V is 1}, for each $W=W_f$ with $r(W)>0$ we
obtain a well-defined $1$-dimensional subspace $V$ of $\C^2$. Hence our 
concrete description of the map $\pi$ of~\eqref{eq:definition pi} is
given by 
\begin{equation}\label{eq:definition pi concrete} 
\pi(W) = V \in \PP^1
\end{equation}
where we view the one-dimensional subspace $V$ of $\C^2$ as an element
in $\PP^1$ as usual. 

We now use the map $\pi$ to define a homomorphism $\lambda_W: S^1 \to
SU(2)$ associated to $W$. 
First consider the case $r=r(W)>0$.
Let $v\in S^3\subset\C^2$ be a representative for $\pi(W)=V$ and choose
$u\in S^3$ such that $u\perp v$.
The corresponding homomorphism $\lambda_W:S^1\to SU(2)$ is defined by
$\bigl(\lambda_W(z)\bigr)(u)=z^{r(W)}u$ and
$\bigl(\lambda_W(z)\bigr)(v)=z^{-r(W)}v$.
More concretely, when written in the $u,v$-basis we have
\begin{equation}\label{eq:matrix for lambda}
\lambda_W(z)=\begin{pmatrix}z^{r(W)}&0\cr0&z^{-r(W)}\end{pmatrix}.
\end{equation}
(Note the elements $u$ and~$v$ are determined by~$W$ only up to multiplication by an
element of~$S^1$, but the resulting homomorphism $\lambda_W$ is independent of
these choices.)
In the case $r(W)=0$ we simply define 
$\lambda_W(z)\equiv 1$ to be the trivial homomorphism taking every element to the
identity in $SU(2)$. 

More detailed information about subspaces of rank~$r$ is given in the
following proposition.

\begin{proposition}\label{kerbasis}
Let $W=W_f$ with $r(W_f) = r > 0$. Then
\begin{enumerate}
\item
We have $r = -  \min \{k \mid \mbox{$W_f$ has an element of degree
  $k$} \} $. 
\item
A basis for the kernel of the orthogonal projection $W_f\to \calK_+$ is given
by the set
$$\{x, zx, z^2 x, \dots, z^{r-1} x\}$$
where $x\in W_f$ satisfies $\deg (x) = -r$.
The subspace $V$ of Lemma~\ref{lemma:dim V is 1} is spanned by $z^{r-1}x$.
\item 
The orthogonal projection from $W_f$ to $\lambda_W(\calK_+)$ is an isomorphism.
\end{enumerate}
\end{proposition}

\begin{proof}
Recall that $a(z)\mapsto f(z)a(z)$ gives an isomorphism $\calK_+\to W_f$
and in particular it gives an isomorphism from $\Ker (M_f)_{++}$ to
the kernel~$K$  of the orthogonal projection $W_f \to \calK_+$.

Since
$$\dim K=\dim \Ker(M_f)_{++}=r> 0,$$
there must exist elements in~$W_f$ having negative degree (zero Taylor part).
If $\deg(y)=-m<0$ then $\{y, zy, z^2 y, \dots, z^{m-1} y\}$
are linearly independent elements of~$K$, so $m\le r$. Thus the set
\begin{equation}\label{set of degrees}
\min \{k \mid \mbox{$W$ has an element of degree $k$} \}
\end{equation}
is bounded below. 
Now let $x$ be an element of $W$ with degree equal to the minimum
of~\eqref{set of degrees}
and let $m := - \deg(x)$. 
Then $0<m\le r$.
Since $\deg(x)$ is minimum, $x\ne zy$ for any $y\in W$.
Consider the set 
$$B:=\{x, zx, z^2 x, \dots, z^{m-1} x\}\in K.$$
We claim that $B$ forms a basis for the kernel $K$. 
Suppose for a contradiction there exists $w \in K$ which is not in the linear span of~$B$.
By multiplication by powers of $z$, we may assume without loss of
generality that $\deg (w) = -1 $.
Using that $x\ne zy$ for any~$y$, we see
that $\{ z^{m-1} x, w \} $ are linearly independent
in~$V:=\langle z^{-1}  \rangle$, contradicting Lemma~\ref{lemma:dim V is 1}.
Therefore there is no such $w$, and so $B$ is a basis for~$K$.
In particular,
$$r=m= -\min \{k \mid \mbox{$W$ has an element of degree $k$} \}.$$
This establishes the first two parts of the proposition.

Part~(2) tells us that $W$ contains no elements of negative degree outside
of the linear span of~$B$, and since $\dim \CoKer \bigl((M_f)_{++}\bigr) = r$,
it follows that the least degree of any element of~$W_f$ outside of the closed
$\CC[z]$ module generated by~$x$ is~$r$.
In other words, there exists $y$ with $\deg(y)=r$ such that $W=W_{x,y}$
and so orthogonal projection $W_f$ to $\lambda_W(\calK_+)$ is an isomorphism.
\end{proof}

\begin{proposition}\label{rankF2r}
Suppose $W \in F_{2r}:=\sgrzrsw $.
Then $ r(W) \le r $  and $r(W)=r$ if and only if
$W\in F_{2r}\setminus F_{2r-2}$.
\end{proposition}

\begin{proof}
$W  \in F_{2k} \setminus F_{2k-2}$ for some $k \le r$.
By Lemma~\ref{lemma:unique w} there exists $w \in W$ such that
$$w = z^{-k} u_0 + \ldots + z^{k-1} u_{2k-1}.$$
Set $u = u_0$ and choose $v \perp u$ .

Let $  x $  be the element of least degree in $W$.
By inspection of the form of $W$, $ x = z^j w + c  z^k v $
for some $c\in \CC$ and some~$j$.
Hence $- \deg(x)\le k\le r.$
But  then by Proposition~\ref{kerbasis}, $- \deg(x) = r(W). $
\end{proof}

We also need the following notation. 
For $\lambda$ a homomorphism $\lambda: S^1
\to SU(2)$, we also view $\lambda$ as an element of $\Omega SU(2)$. We
let $\calK_\lambda$ denote the subspace
$\lambda(\calK_+)$. 
Let $\mathcal{O}$ denote the ring of infinite series $a(z) =
\sum_{n=0}^\infty a_n z^n$ in non-negative
powers of $z$ which converge on the closed unit disk $D^2$ in~$\C$. (In
particular, by assumption such $a(z)$ are holomorphic on the interior
of the unit disk.) By
slight abuse of notation we sometimes view an element $a(z)$ in
$\mathcal{O}$ as a function on the boundary $S^1$, while at other times
we view it as a function on $D^2$. 

Following \cite{PS86} we also introduce the following sets of
matrix-valued functions. First let 
\[
\mathcal{N}^- 
:=\left\{
\begin{pmatrix}
1+z^{-1}a(z^{-1})&b(z^{-1})\cr z^{-1}c(z^{-1})&1+z^{-1}d(z^{-1})\cr
\end{pmatrix}
\Big\vert \ a(z), b(z), c(z), d(z)\in\mathcal{O}
\right\}
\]
be the set of $2\times 2$ matrix-valued functions $A(z)$, where the
matrix entries are of the above form (and in particular are
holomorphic on the region $\{\|z\|>1\}$) and such that $A(\infty)$ is
upper-triangular with $1$'s on the diagonal. 
Restricting this set slightly further we also define 
\[
N^-:=
\{A(z)\in{\mathcal N}^-\mid A(z) \mbox{ is invertible for all~$z$}
  \textup{  with } \|z\|\geq 1 \}
\]
and, restricting still further, we set 
\[
L^-_1:= \left\{A(z)\in N^-\mid A(\infty)=\begin{pmatrix} 1 & 0 \\ 0 &
  1 \end{pmatrix} \right\}.
\]
The definition of $L^-_1$ in particular implies that elements in~$L^-_1$
have the form 
\[
\begin{pmatrix}
1+z^{-1}a(z^{-1})&z^{-1} b(z^{-1})\cr z^{-1}c(z^{-1})&1+z^{-1}d(z^{-1})
\end{pmatrix}
\]
i.e. the constant term in the upper-right corner must be equal to
$0$. 

Extending the notation of Section~\ref{sec:grassmannian} slightly, for
$A(z): S^1 \to GL(2,\C)$ any polynomial loop, \emph{not} necessarily
based at the identity, we denote by $M_A: \mathcal{K} \to \mathcal{K}$
the multiplication operator defined by $M_A(h)(z) := A(z) \cdot h(z)$
and let $$W_A := \overline{M_A(\mathcal{K}_+)}$$ denote the closed subspace
of~$\mathcal{K}$ which is the closure of the image of $\mathcal{K}_+$.
More concretely, if we let $u=u(z):S^1 \to \C^2$ and $v=v(z): S^1 \to \C^2$ denote the
first and second columns of $A$ respectively, then 
$W_A$ is the closure of the span of the elements
in $\mathcal{K} := \mathcal{H} \otimes\C^2$ of the
form $$\{z^k u(z), z^k v(z)\mid k \geq 0\}.$$  Motivated by
this, given $2$ vector-valued functions $u(z), v(z): S^1 \to \C^2$
which are everywhere linearly independent, we also denote $$W_{u,v} := W_A$$
where the matrix $A := [u \hsm v]$ is obtained by putting $u(z)$ in
the left column and $v(z)$ in the right column. 

For any homomorphism $\lambda: S^1 \to SU(2)$ there
exists an orthonormal basis $u_\lambda, v_\lambda$ of $\C^2$
with respect to which
$\lambda(z)$ is diagonal with  $\lambda(z) = 
{\diag} (z^r, z^{-r})$ for some
$r\geq 0$.
The integer~$r$ uniquely determined by~$\lambda$ and for $r>0$ the
basis~$\{u_\lambda,v_\lambda\}$ is uniquely determined up to common scalar
multiple. 

Multiplying the matrices gives
\[
\displaylines{ 
\lambda L_1^- \lambda^{-1} = 
\hfill\cr\hfill
\left\{
\begin{pmatrix}
1+z^{-1}a(z^{-1})&z^{2r-1}b(z^{-1})\cr
z^{-2r-1}c(z^{-1})&1+z^{-1}d(z^{-1})\cr
\end{pmatrix}
\Big\vert
\textup{ invertible for } \|z\|\geq 1 \textup{ and } 
\ a(w), b(w), c(w), d(w)\in\mathcal{O}
\right\}.\cr
}
\]
Following Pressley and Segal,~\cite[8.6.3(iv)]{PS86} , we now define 
\begin{equation}\label{eq:definition U lambda} 
U_\lambda := \lambda L_1^- \lambda^{-1} \mathcal{K}_\lambda = \{ W_A
\hsm \vert \hsm A(z) \in \lambda L_1^- \}
\end{equation}
where here we view a $2 \times 2$ matrix as a linear transformation on
$\C^2$ written with respect to the basis $\{u_\lambda, v_\lambda\}$, and
$W_A$ denotes the closed subspace $M_A(\mathcal{K}_+)$ defined
above. More concretely, $U_\lambda$ consists of closed subspaces
$W_{u,v}$ in $\mathcal{K}$ where $u=u(z)$, $v=v(z)$ are of the form 
\[
u(z) = \begin{pmatrix} z^r (1+z^{-1} a(z^{-1})) \\ z^{-r-1}
  c(z^{-1}) \end{pmatrix}, 
\quad 
v(z) = \begin{pmatrix} z^{r-1} b(z^{-1}) \\ z^{-r} (1 + z^{-1}
  d(z^{-1})) \end{pmatrix}, 
\]
where both $u$ and $v$ are written with respect to the basis ${u_\lambda,
  v_\lambda}$, and $a(z), b(z), c(z), d(z) \in \mathcal{O}$, and 
\[
\begin{pmatrix} 
1+z^{-1}a(z^{-1}) & z^{2r-1} b(z^{-1}) \\ 
z^{-2r-1} c(z^{-1}) & 1+z^{-1}d(z^{-1})
\end{pmatrix} 
\]
is invertible for $z$ with $\|z\|\geq 1$. (We will give an alternative,
and more conceptual, description of $U_\lambda$ below.) 

We will also need to analyze the following subset of
$U_\lambda$. Namely, we define 
\begin{equation}\label{eq:def Sigma lambda}
\Sigma_\lambda := \left\{ W_{u,v} \in U_\lambda \hsm \Big\vert \hsm 
u(z) = \begin{pmatrix} z^r \bigl(1+z^{-1} a(z^{-1})\bigr) \\ z^{-r-1}
  c(z^{-1}) \end{pmatrix} \textup{ and } 
v(z) = \begin{pmatrix} z^{-r} b(z^{-1}) \\ z^{-r} \bigl(1 + z^{-1}
  d(z^{-1})\bigr) \end{pmatrix}   \right\}.
\end{equation}
In other words, if $r>0$ then $\Sigma_\lambda$ consists of those subspaces in
$U_\lambda$ which can
be expressed as $W_{u,v}$ where the $u_\lambda$ coordinate of $v$ has
no non-zero coefficients for $z^{-\ell}$ for $\ell < r$. 
Note that if $W \in \Sigma_\lambda$, then
$r(W)=r$ since it can be seen from the definition to contain an element
of degree $-r$ but none of lower degree.

For a homomorphism $\lambda: S^1 \to SU(2)$, we let $|\lambda| \geq 0$ denote
the unique non-negative integer such that $\lambda(z) = \textup{diag}
(z^{|\lambda|}, z^{-|\lambda|})$ with respect to some (orthonormal)
basis. 
For a fixed integer $r \geq 0$ we now define 
$$\Sigma^G_r:=\bigcup_{|\lambda|=r}\Sigma_\lambda,$$
i.e. $\Sigma^G_r$ is the $G$-orbit of $\Sigma_\lambda$. 
Similarly let 
\[
U^G_r := \bigcup_{|\lambda|=r}U_\lambda.
\]
These spaces play the roles analogous to that of 
$\Sigma_\lambda$ and $U_\lambda$, respectively, in the arguments of
Pressley-Segal. 

\begin{remark} 
(This is a technical remark for readers intending to work
explicitly with these spaces $U^G_r$.)

If $W\in U^G_r$ then $r$ is not uniquely determined by~$W$.
Indeed, let $e,f$ be the standard basis for~$\C^2$, 
$r=2$, and consider $W = W_{z^2e, ze+z^{-2}f}$.
Then the only $\lambda$ with $|\lambda|=2$ for which $W\not\in U_\lambda$
is $\lambda e=z^{-2}e$, $\lambda f=z^2 f$, corresponding to the 
ordered orthonormal basis~$f,e$.
However we can also express this same subspace as 
$W= W_{ze+z^{-2}f, z^{-1}f}$, which exhibits $W$ as
an element of~$U^G_1$.
As we shall see later, $r(W)$ is the least $r$ such that $W\in U_r^G$.
\end{remark}

We now proceed to an analysis of the topology of
$U^G_r$ and $\Sigma^G_r$.
We first show that $\Sigma^G_r$ is 
$G$-homotopy equivalent to $\PP^1$. We then show that $U^G_r$ can be
regarded as the total space of a rank $2r-1$ complex vector bundle
over $\Sigma^G_r$. In fact we are able to identify the bundle
explicitly as the pullback $\pi^*(\tau^{2r-1})$, where $\pi$ is the
map to $\PP^1$ defined above and $\tau$ is the
tangent bundle to $\PP^1$.
Our main geometric
statement is Theorem~\ref{Ur_bundle},
which leads to the homotopy equivalence
$S{\Gr'}_{\sandwich,r}^z\into S\Gr_{\apsm,r}^z$
of Theorem~\ref{theorem:milnorlimit}
and ultimately to the homotopy equivalence
of Theorem~\ref{theorem:poly into psm}.

Fix $r>0$. For $x\in\PP^1$, choose a unit vector $v \in \C^2$
representing the line $x$, and also choose a unit vector $u$ such that
$u,v$ form the left and right columns respectively of an element of $SU(2)$.
Define $s_r(x)\in\Sigma^G_r$ by
$$s_r(x):=W_{z^ru,z^{-r}v}.$$
This is well-defined since the subspace $W_{z^ru, z^{-r}v}$ is
independent of the choices made for $u$ and $v$. 
It is straightforward to check that $s_r:\PP^1\to\Sigma^G_r$
is $G$-equivariant, and also that 
$\pi\circ s_r=1_{\PP^1}$. 

Notice that $L_1^-$ is contractible, with an explicit contraction given
by $H_t(A)(z):=A(t^{-1}z)$.
This leads to the following proposition, which is the $G$-equivariant
analogue of the fact, recorded in \cite[Theorem 8.6.3]{PS86}, that
$\Sigma_\lambda$ is contractible in the non-equivariant
setting. 

\begin{proposition}\label{prop: pi is a homotopy equivalence}
The map $\pi:\Sigma_r^G\to\PP^1$ is a $G$-homotopy equivalence for all~$r>0$,
with $G$-homotopy inverse~$s_r$.
\end{proposition}

\begin{proof}
It is straightforward from its definition that $\pi$ is $G$-equivariant.
The $G$-homotopy coming from $H_t(A)(z):=A(t^{-1}z)$ is given explicitly
as follows.
Given an element $W_{u,v} \in \Sigma_\lambda$ for $u(z),
v(z)$ of the form given in~\eqref{eq:def Sigma lambda}, 
 we can define 
\[
u_t = u(t,z) := \begin{pmatrix} z^r (1+ t z^{-1} a(tz^{-1})) \\  (t z^{-1})^{r+1}
  c(t z^{-1}) \end{pmatrix}, 
\quad 
v_t = v(t,z) := \begin{pmatrix} (t z^{-1})^r b(t z^{-1}) \\ z^{-r} (1 +  t z^{-1}
  d(t z^{-1})) \end{pmatrix}  
\]
and consider the corresponding subspaces $W_{u_t, v_t}$. This
evidently defines a $G$-equivariant deformation retraction taking
$W_{u,v}$ to $s_r(\pi(W_{u,v}))$, as desired. 
\end{proof}

In the case $r=0$, there is only one homomorphism $\lambda: S^1 \to
SU(2)$ with $r(\lambda)=0$. Therefore $\Sigma^G_0 \cong
\Sigma_\lambda$ where $\lambda$ is the trivial homomorphism.
Thus the
next statement follows from the contraction
$H_t(A)(z):=A(t^{-1}z)$ of $L_1$ in the same way.

\begin{proposition}\label{prop:contractible}
The space $U^G_0=\Sigma_0^G$ is $G$-equivariantly contractible. 
\end{proposition}

The next theorem is our main technical geometric result. It identifies
$U^G_r$ as the total space of a complex vector bundle over
$\Sigma^G_r$ obtained by pullback via the $G$-homotopy equivalence
$\pi: \Sigma^G_r \to \PP^1$ discussed above. Recall that $\tau$
denotes the tangent bundle to $\PP^1$.

\begin{theorem}\label{Ur_bundle}
Let $r>0$. Then the space $U_r^G$ is $G$-homeomorphic to the total space of
the bundle $\pi^*(\tau^{2r-1})$ over~$\Sigma_r^G$.
\end{theorem}

\begin{proof}
Following the notation from Section~\ref{sec:filtration}, the total
space~$E(\tau^{2r-1})$ of the bundle $\tau^{2r-1}$ over $\PP^1$ can be
described as
$$E(\tau^{2r-1})=
\bigl\{(v,x)\mid v\in S^3\subset\C^2, x\in (v^\perp)^{2r-1}\bigr\}
/\mathord{\sim}$$
where $(v,x)\sim(\zeta v,\zeta x)$ for $\zeta\in S^1$.
Thus the total space of the pullback bundle $\pi^*(\tau^{2r-1})$ over
$\Sigma^G_r$ is 
$$
E\bigl(\pi^*(\tau^{2r-1})\bigr)=
\bigl\{(W,v,x)\mid W\in\Sigma_r^G, v\in S^3 \mbox{ with $[v]=\pi(W)$},
x\in (v^\perp)^{2r-1}\bigr\}
/\mathord{\sim}$$
where $(v,x)\sim(\zeta v,\zeta x)$ for $\zeta\in S^1$.

We now explicitly define a map $\phi: E(\pi^*(\tau^{2r-1})) \to
U^G_r$, which we later show is a $G$-equivariant
homeomorphism. Let $\lambda$ be a homomorphism with $r(\lambda)=r$ and
let $X=[W,v,x]\in E\bigl(\pi^*(\tau^{2r-1})\bigr)$
with $W\in\Sigma_{\lambda} \subseteq \Sigma^G_r$.
Write $x=(a_0u,a_1u,\ldots ,a_{2r-2}u)$ where $u\perp v$ and $a_j\in\C$.
Since $\pi(W)=[v]$, the homomorphism $\lambda=\lambda(W)$ is given by
$\lambda(z)=\begin{pmatrix} z^r & 0\\0 & z^{-r} \end{pmatrix}$ written in the $u,v$ basis. 
Let $$e(z)=a_0+a_1z+\ldots +a_{2r-2}z^{2r-2}. $$
Since $W\in\Sigma_\lambda$ we can write
$W = W_{u,v}$ where 
\[
u = u(z) = \begin{pmatrix} z^r \bigl(1+ z^{-1} a(z^{-1})\bigr) \\  z^{-r-1}
  c(z^{-1}) \end{pmatrix}, 
\quad 
v = v(z) = \begin{pmatrix}  z^{-r} b(z^{-1}) \\ z^{-r} \bigl(1 +   z^{-1}
  d(z^{-1})\bigr) \end{pmatrix}  
\]
for some $a(w),b(z),c(w),d(w)\in\calO$
and where the right hand sides are written with respect to the ordered basis
$u,v$. 
We now explicitly define $\phi(X = [W,v,x]) \in U^G_r$ as follows. 
Let $P=AE$ where 
$$A=\begin{pmatrix}
1+z^{-1}a(z^{-1})&b(z^{-1})\cr z^{-2r-1}c(z^{-1})&1+z^{-1}d(z^{-1})\cr
\end{pmatrix}
$$
for the $a,b,c,d$ are the elements in $\calO$ above 
and
$$E=\begin{pmatrix}
1&ze(z)\cr0&1\cr
\end{pmatrix}$$
(all written in the $u,v$ basis). 
Define 
\begin{equation}\label{eq:def phi} 
V = \phi(W=W_{u,v}) :=W_{P(z^ru),P(z^{-r}v)}.
\end{equation}
Multiplying the matrices $A$ and~$E$ shows that the subspace $V$ thus defined
is an
element of $U^G_r$. 
Next we check that the construction of $\phi(W)=V$ given above is
independent of the choices made. 
Suppose $X = [W, v', x']$ and suppose 
$u'$ is orthogonal to $v'$. 
Then $u'=\zeta_1u$, $v'=\zeta_2v$, and $x'=\zeta_2x$ for
some $\zeta_1, \zeta_2\in S^1.$
In the construction given above we then obtain
$e'(z)=\zeta_1^{-1}\zeta_2e(z)$ instead of $e(z)$. 
In turn, $E$ is replaced by
$$E'=\begin{pmatrix}
1&\zeta_1^{-1}\zeta_2ze(z)\cr0&1\cr
\end{pmatrix}
=Z^{-1}EZ$$
where
$$Z=\begin{pmatrix}
\zeta_1&0\cr0&\zeta_2\cr
\end{pmatrix}.$$
It is also straightforward to compute that 
the matrix~$A'$ which replaces~$A$ 
is $A'=Z^{-1}AZ$.
Therefore $P$ gets replaced by $P':=A'E'=Z^{-1}PZ$, and we obtain 
$V' = W_{P'(z^ru'),P'(z^{-r}v')}
=W_{\zeta_1P(z^ru),\zeta_2P(z^{-r}v)}$, which is equal to $V$. 
Hence $\phi$ is well-defined. 

The fact that $\phi$ is a bijection follows from solving equations to find
$A$ and~$E$ from~$P$ as in  the proof of~\cite[Equation 8.4.4]{PS86}.
This gives a fibrewise inverse to
$\phi$. 
The map $\phi$ is also $G$-equivariant since by definition, the action of $G$ on $\calK = \calH \otimes \CC^2$
is via the standard action of $G$ on the second factor. Hence 
\[
\phi(g \cdot X) = \phi([gW, gv, gx]) = W_{P(z^r(gu)), P(z^{-r}(gv))} =
g \cdot V,
\]
as desired. 
Finally, the topology on $U_G^r$ is defined as a quotient of a subspace
of~$B(\calK)$, the bounded linear operators on~$\calK$,
where two operators are equivalent if they define the same subspace.
A map from a quotient space is continuous if and only if the 
composition with the quotient map is continuous, and the 
latter is given by matrix multiplications.
Thus $\phi$ is continuous. The same argument applies to $\phi^{-1}$.
Hence $\phi$ is a $G$-equivariant homeomorphism.
\end{proof}

The explicit description of $U^G_r$ as a total space of a bundle in
the previous theorem is a key tool that 
allows us to
show our main theorem of this section (Theorem~\ref{theorem:milnorlimit}) that the inclusion of a certain subspace
$S{\Gr'}_{\sandwich,r}^z$ (defined precisely in~\eqref{bdd:defns})
into 
$S\Gr_{\apsm,r}^z$ is a $G$-homotopy
equivalence.
However, we must first analyze more closely
the relation between the spaces $U^G_r$ and
the spaces $S\Gr^z_{\apsm,r}(\calK)$ discussed in previous sections. 
This requires a new description of the spaces $U_\lambda$ and $\Sigma_\lambda$,
which we will initially denote as $\tilde{U}_\lambda$
and $\tilde{\Sigma}_\lambda$. 
(In Proposition~\ref{sigmathm} and 
Corollary~\ref{cor:tildeU equals U} we show that in fact the two
descriptions yield the same spaces.) 
Specifically, define
$$
\begin{aligned}
&\tildeUlambda := \{ W \in\sgrz\mid \mbox{the orthogonal projection from
$W$ to~$ \calK_\lambda$ is an isomorphism} \},\cr
&\tildeUr:=\cup_{|\lambda|=r}\tildeUlambda,\cr
&\tildeSigmalambda := \{ W \in \tildeUlambda \mid r(W)=|\lambda|\},
\textup{ and } \cr
&\tildeSigmar:=\cup_{|\lambda|=r}\tildeSigmalambda \cr
\end{aligned}
$$
corresponding to the spaces Pressley-Segal denote as $U_S$ and $\Sigma_S$
in \cite[pages~103 and~107]{PS86}.

Before proceeding we sketch the overall plan of the remainder of the 
(rather technical) argument leading to
Theorem~\ref{theorem:milnorlimit}. 
First we prove that $\tildeSigmar=\Sigma_r^G$ and $\tildeUr=U_r^G$.
We then use the new descriptions of the spaces $\Sigma_r^G$ and
$U_r^G$ to show that 
$U^G_r\cap S\Gr^z_{\apsm,r-1}(\calK)=U^G_r\setminus\Sigma_r^G$
and that $U^G_r\cup S\Gr^z_{\apsm,r-1}(\calK)=\sgrz$.
We also explicitly identify 
$U^G_r\setminus\Sigma_r^G$ with the complement of the zero cross-section
of $\pi^*(\tau^{2r-1})$.
Then, 
repeating the arguments thus far for the intersections of the relevant
spaces with the subspaces $\sgrzsw$ of
bounded weight, we obtain a description of $U^G_r\cap\sgrzsw$
as the total space of the pullback of~$\tau^{2r-1}$ to $\Sigma^G_r\cap\sgrzsw$,
with $(U^G_r\cap S\Gr^z_{\apsm,r-1}(\calK))\cap\sgrzsw$ as the complement
of the zero cross section. 
Finally, we use $G$-homotopy equivalences on the total spaces and complements of the
zero cross-sections of the bundles (induced by a $G$-homotopy equivalence
$\Sigma^G_r\cap\sgrzsw\to\Sigma^G_r$ of the base spaces) as part of an
induction argument to show that
$$S{\Gr'}_{\sandwich,r}^z:= S{\Gr}^z_{\sandwich}(\calK)\cap \sgrzr\into\sgrzr$$
is a homotopy equivalence for each~$r$.

With this broad outline in place, we proceed to the details of the
argument. 

\begin{lemma}\label{Ucontainment}
Let $U_\lambda$ and $\tildeUlambda$ be as defined above. Then
$U_\lambda \subset \tildeUlambda$. 
\end{lemma}

\begin{proof}
Let $W = W_{u, v}.$
Orthogonal projection takes $u(z)$ to $z^r u_\lambda$ because it sends
to $0$ the multiples of 
the first basis element $u_\lambda$  by $z^k$ for $k<r$. 
It takes $v(z)$ to $z^{-r} v_\lambda$ since it sends to $0$ the
multiples of 
the second basis element $v_\lambda$ by $z^k$ for $k<-r$. 
Since both the domain and range of the projection is a free rank $2$ module over $\CC[z] $
and we have just shown that the map takes generators to generators,
it is an isomorphism.
\end{proof}

\begin{lemma}\label{rkUrbound}
If $W \in \tildeUr $ then $r(W) \le r$.
\end{lemma}

\begin{proof}
Suppose $W \in \tildeUr$.
Then $W\in \tildeUlambda$ for some $\lambda$ with $|\lambda|=r$.
If $x\in W$ with $\deg(x)<-r$, then the orthogonal projection from $W$ to
$\calK_\lambda$  takes $x$ to~$0$, which is impossible since this projection
is required to be an isomorphism.
Thus $W$ has no elements of degree $< -r$ and hence $r(W) \le r$ as desired.
\end{proof}

\begin{lemma}\label{rkSigmar}
If $W \in \Sigma_r^G $ then $r(W) = r$.
In particular, $\Sigma_r^G\subset \tildeSigmar$.
\end{lemma}

\begin{proof}
Suppose $W\in \Sigma_r^G$.
By inspection, $W$ contains an element of degree~$-r$ (namely, the $v(z)$ from the
definition), so $r(W)\ge r$.
But $W\in \Sigma_r^G\subset U_r^G\subset \tildeUr$, so $r(W) \le r$.
Thus $W\in\Sigma_r^G$ implies $r(W)=r$.
\end{proof}

We also include two technical lemmas about holomorphic functions to be
used in the proof of the proposition below. 

\begin{lemma}\label{holobndry}
Let $h(z):S^1\to\CC$ be piecewise smooth.
Suppose that the coefficient of $z^k$ in the Fourier expansion of $h$ is
zero for $k<0$.
Then $h\in\calO$.
\end{lemma}

\begin{proof}
Let $\sum_{k=0}^\infty c_k z^k$  be the Fourier expansion of~$h(z)$, where
$c_k\in \CC$.
Since $h$ is piecewise smooth, the Fourier expansion of $h(z)$ converges
to~$h(z)$.
Since the series $\sum_{k=0}^\infty c_k z^k$ converges for all $z$
with $|z|=1$, its radius of convergence is greater than~$1$ so it defines
a holomorphic function on the unit disk whose boundary value is~$h(z)$.
\end{proof}

\begin{lemma}\label{nullhomo}
Let $h(z)$ be holomorphic on a domain containing~$D^2$.
Suppose that the restriction of $h(z)$ to $S^1$ is never~$0$ and that
$h|_{S^1}:S^1\to\CC\setminus\{0\}$ is null homotopic.
Then $h(z)$ has no zeros in $D^2$.
\end{lemma}

\begin{proof}
Consider the curve $\gamma(z):=h(S^1)\subset\CC$.
By hypothesis, $\gamma$ is null homotopic.
According to the Argument Principle
$$\#\mbox{ of zeros of $h(z)$ on $D^2$}=\int_{S^1}\frac{h'(z)}{h(z)}\,dz
=\int_\gamma \frac{1}{w}\,dw
=\mbox{winding $\#$ of $\gamma$ about the origin}=0.$$
Hence the origin is not in~$h(D^2)$.
\end{proof}

We are now in a position to prove the equivalence of our two definitions
of $\Sigma_r^G$, corresponding to~\cite[Prop.8.4.1]{PS86}.

\begin{proposition}\label{sigmathm}
Let $\tildeSigmar$ and $\Sigma_r^G$ be as defined above. Then
$\tildeSigmar=\Sigma_r^G$. 
\end{proposition}

\begin{proof}
The containment $\Sigma_r^G\subset\tildeSigmar$ is the content of
Lemma~\ref{rkSigmar}.
For the other containment, suppose 
$W\in \tildeSigmar$.
Then $W\in \tildeSigmalambda$ for some $\lambda$ with $|\lambda|=r$.
Let $x(z)\in W$ have degree~$-r$.
Then $$x(z)=z^{-r}b(z^{-1})u_\lambda + z^{-r} e(z^{-1})v_\lambda$$ for
$b(w), e(w) \in \mathcal{O}$. 
The orthogonal projection  $W\to \calK_\lambda$ (which is
an isomorphism since $\tildeSigmalambda\subset \tildeUlambda$)
takes $x(z)$ to $e_0 v_{\lambda}$, where $e_0$ is the constant term of~$e(z^{-1})$. Hence $e_0 \ne 0$.
Set $v(z):=x(z)/e_0$ and let $u(z)$ be the inverse image of $z^r u_\lambda$
under the projection $W \to \calK_\lambda$. 
Since the orthogonal projection is an isomorphism, it follows that $W=W_{u,v}$
and this exhibits $W$ as an element of $\Sigma_\lambda$ as in~\eqref{eq:def Sigma lambda}, provided the holomorphicity and
invertibility conditions are satisfied.
Applying Lemma~\ref{holobndry} to the components of $z^{-r}u(z^{-1})$
and $z^rv(z^{-1})$ shows that they are boundary values of holomorphic
functions on $|z|>1$.
To see invertibility, let $A(z)\in \GL(2)$ be the matrix whose columns are formed from $z^{-r}u(z)$
and~$z^rv(z)$ and let $d(z)=\det A(z)$.
According to Lemma~\ref{beta:preparation}, the homotopy class of the function
$z\to d(z)$ is $$-2\Ind(W_f)=0\in \pi_1(\CC\setminus\{0\}\cong\ZZ.$$
Applying Lemma~\ref{nullhomo} to $d(z^{-1})$ shows that $d(z^{-1})$ is never
zero on~$|z|\ge 1$.
Thus $W\in\Sigma_\lambda \subseteq \Sigma_r^G$. 
\end{proof}

\begin{corollary}\label{stratification}
The spaces $\{\Sigma_r^G\}$ form a stratification of $\sgrz $ and $W\in \Sigma_r^G$ if
and only if the rank of $W$ is~$r$.
\end{corollary}
\begin{proof}
Suppose $W\in\sgrz $.
By Proposition~\ref{kerbasis} part (3), we know that $W\in
\tilde{U}_{\lambda_W}$. On the other hand, by definition of the
homomorphism $\lambda_W$ (see~\eqref{eq:matrix for lambda}) we know
$|\lambda_W| = r(W)$, so by definition of $\tilde{\Sigma}_{\lambda}^G$
we conclude $W \in \tilde{\Sigma}_{\lambda}^G$. 
By Proposition~\ref{sigmathm} this implies $W \in
\Sigma^G_{r(W)}$. Since each element $W$ of $\sgrz$ has a unique rank
we conclude the $\Sigma^G_{r(W)}$ form a stratification of $\sgrz$, i.e.
$$\sgrz=\coprod_r \Sigma_r^G.$$
\end{proof}

We also get as a consequence the equivalence of the two definitions of $U_r^G$.

\begin{corollary}\label{cor:tildeU equals U}
Let $\tildeUr$ and $U_r^G$ be as defined above. Then
$\tildeUr=U_r^G$. 
\end{corollary}

\begin{proof}
The assertion that $U_r^G\subset\tildeUr$ is the content of Lemma~\ref{Ucontainment}.
Conversely, suppose $W\in \tilde{U}_\lambda$ with $|\lambda|=r$.
The orthogonal projection $W\to \calK_\lambda$ is an
isomorphism, and therefore
there exist unique $u(z), v(z)\in W$ projecting to $z^ru_\lambda$ and $z^{-r}v_\lambda$ respectively.
Regarding $W$ as an element of $\tilde{\Sigma}_{r(W)}^G=\Sigma_{r(W)}^G$
shows, as in the proof of Proposition~\ref{sigmathm}, that $z^{-r}u(z^{-1})$ and $z^rv(z^{-1})$
are boundary values of holomorphic functions on a domain containing $|z|\ge1$
and that the matrix whose columns are formed from these functions is
invertible in~$|z|\ge1$.
Thus the functions $u(z), v(z)$ exhibit $W$ as an element of~$U_r^G$.
\end{proof}

With the aid of our alternate descriptions of $U_r^G$ and $\Sigma_r^G$,
we can now relate $\sgrz$ to our bundle description of $U_r^G$.

\begin{proposition}
Under the $G$-equivariant identification $\phi: E(\pi^*(\tau^{2r-1})) \to U^G_r$, the intersection $$U^G_r\cap S\Gr^z_{\apsm,r-1}(\calK)$$
is identified with the complement of the zero cross-section
of $\pi^*(\tau^{2r-1})$.
\end{proposition}

\begin{proof}
Note that the inclusion $\Sigma_r^G\subset U_r^G$ corresponds, under the
identification $\phi$, with the inclusion of the zero cross-section into the
total space.
Suppose $W\in U_r^G \cap S\Gr^z_{\apsm,r-1}(\calK) $.
Since $r(W)<r$, Lemma~\ref{rkSigmar} implies that $W\notin \Sigma^G_r$.
That is, $W$ lies in the complement of the zero cross-section
of~$\pi^*(\tau^{2r-1})$.
Conversely, if $W\in U_r^G$ is not in $\Sigma^G_r$, by Lemma~\ref{rkSigmar}
its rank cannot be~$r$ and therefore it lies
in $U_r^G \cap S\Gr^z_{\apsm,r-1}(\calK) $.
\end{proof}

We also record the following, which again makes use of our alternate
descriptions of $U_r^G$ and $\Sigma_r^G$.

\begin{proposition}\label{pushout:apsm}
We have $\sgrzr =U_r^G\cup S\Gr^z_{\apsm,r-1}(\calK)$.
\end{proposition}

\begin{proof}
Suppose $W\in U_r^G$.
By Lemmas~\ref{Ucontainment} and~\ref{rkUrbound}, $r(W)\le r$
so $W\in S\Gr^z_{\apsm,r}(\calK)$.
Therefore $U_r^G\subset S\Gr^z_{\apsm,r}(\calK)$, while
$S\Gr^z_{\apsm,r-1}(\calK) \subset S\Gr^z_{\apsm,r}(\calK)$ is trivial.

Conversely, suppose $W\in S\Gr^z_{\apsm,r}(\calK)$.
If $r(W)<r $ then $W\in S\Gr^z_{\apsm,r-1}(\calK)$ while if $r(W)=r$
then $W\in \Sigma_r^G\subset U_r^G$.
\end{proof}

We now define the subset $S{\Gr'}^z_{\sandwich,r}(\calK)$ referred to
above, which is an important ingredient in our main theorem, 
as well as the bounded versions of the spaces $U^G_r$ and
$\Sigma^G_r$. 

\begin{equation}\label{bdd:defns}
\begin{aligned}
S{\Gr'}^z_{\sandwich,r}(\calK)&:= S{\Gr}^z_{\sandwich}(\calK)\cap \sgrzr\cr
U^G_{\sandwich,r}(\calK)&:= S{\Gr}^z_{\sandwich}(\calK)\cap U^G_r\cr
\Sigma^G_{\sandwich,r}(\calK)&:= S{\Gr}^z_{\sandwich}(\calK)\cap \Sigma^G_r\cr
\end{aligned}
\end{equation}

\begin{proposition}
We have $\sgrzrsw \subset \sgrzrswprime $ and
$\bigcup_r \sgrzrsw = \bigcup_r  \sgrzrswprime .$
\end{proposition}

\begin{proof}
The inclusion $\sgrzrsw \subset \sgrzrswprime $ is a restatement of the fact
that $W \in \sgrzrsw $ implies, according to Proposition~\ref{rankF2r},
that  $r(W)\le r$.
The containment $$\bigcup_r \sgrzrsw \subset \bigcup_r  \sgrzrswprime$$
follows. 
Conversely it is immediate from the definition that
$$\bigcup_r  \sgrzrswprime\subset \sgrzsw :=\bigcup_r \sgrzrsw.$$
\end{proof}

The bounded weight versions of the earlier results are recorded in
Proposition~\ref{prop:summary bounded versions}.
Since the arguments are the same as
those given above (restricted to the appropriate subspaces), we
omit the proofs. 

\begin{proposition}\label{prop:summary bounded versions}
Let $G=SU(2)$. Then:
\begin{enumerate}
\item
The map $\pi:{\Sigma}_{\sandwich,r}^G\to\PP^1$ is a $G$-homotopy equivalence for
all~$r>0$ with homotopy inverse~$s_r$.
\item
The space ${U}^G_{\sandwich,0}=\Sigma_{\sandwich,0}^G$ is
$G$-equivariantly contractible.
\item For $r>0$,  ${U}_{\sandwich,r}^G$ is $G$-homeomorphic to the total
space of the bundle $\pi^*(\tau^{2r-1})$
over~${\Sigma}_{\sandwich,r}^G$
(where $\pi$ refers here to the restriction of $\pi$
to~$\Sigma_{\sandwich,r}^G$).  
\item 
Under the $G$-equivariant identification
$\phi: E\bigl(\pi^*(\tau^{2r-1})\bigr) \to {U}^G_r$, 
the intersection $${U}^G_{\sandwich,r}\cap S{\Gr'}^z_{\sandwich,r-1}(\calK)$$
  is identified with the complement of the zero cross-section
of $\pi^*(\tau^{2r-1})$ over ${\Sigma}_{\sandwich,r}^G$.
\item The space $S{\Gr'}^z_{\sandwich,r}(\calK)$ is the union
${U}^G_{\sandwich,r}\bigcup  S{\Gr'}^z_{\sandwich,r-1}(\calK)$.
\end{enumerate}
\end{proposition}

We are now in a position to prove the main theorem of this section,
Theorem~\ref{theorem:milnorlimit}. 
The basic idea is to make use of our homotopy equivalence
$\Sigma_{\sandwich,r}^G\simeq \Sigma_r^G $ on the base of our bundles to
inductively show, applying a Mayer-Vietoris style argument, 
that $S{\Gr'}_{\sandwich,r}^z\to S\Gr_{\apsm,r}^z$ is a
$G$-homotopy equivalence.

\begin{theorem} \label{theorem:milnorlimit}
The inclusion $S{\Gr'}_{\sandwich,r}^z\to S\Gr_{\apsm,r}^z$ is a $G$-homotopy
equivalence for all~$r$.
\end{theorem}

\begin{proof}
In general, if a topological $G$-space $X$ is a union $U\cup V$,
another $G$-space $X'$ is also a union $U'\cup V'$, and $f:X\to X'$ is
a map of $G$-spaces,
assuming all of the inclusion maps are cofibrations,
then $f$ is a
$G$-homotopy-equivalence if it induces $G$-homotopy-equivalences $U\to U'$,
$V\to V'$ and $U\cap V\to U'\cap V'$.
(See e.g. \cite[Thm.7.1.8]{Sel97} for the non-equivariant version.
Although \cite{Sel97} does not say so explicitly, all the maps constructed and
used there are $G$-equivariant.)
Thus our assertion follows by induction from the comparison of Prop.~\ref{pushout:apsm} with
part (5) of Proposition~\ref{prop:summary bounded versions},
using the fact that the the inclusion of the base space of a $G$-bundle
into the associated total space is always a $G$-homotopy equivalence. 
\end{proof}

\section{Proof of the main theorem}\label{sec:proof}

We are ready to prove the main result, Theorem~\ref{theorem:intro}. We
do this by first showing that 
for $G=SU(2)$ the 
natural inclusion $\Omega_{\poly}G \to
\Omega G$ is an $G$-homotopy equivalence. This reduces the computation
to that of $K^*_G(\Omega_{\poly}G)$ and $K^*_T(\Omega_{\poly}G)$, which was
recorded in Theorem~\ref{theorem:module for poly loops}.

Our approach to the proof that $\Omega_{\poly}G \simeq_G \Omega G$
is similar to that in \cite{HarSel08}, so we keep the explanation brief. 
We use 
$G$-equivariant versions of arguments given by Milnor in 
\cite[Appendix A]{Mil63} to derive general conditions under which a
map is an equivariant homotopy equivalence (it turns out to depend on the map
restricting to equivariant homotopy equivalences on a sequence of
subspaces, the union of which is the whole space). 
As is already pointed out in \cite{HarSel08}, although Milnor does not
make explicit remarks concerning group actions, all the maps
constructed and used in Milnor's proofs
are equivariant.

Let $H$ be a compact Lie group.
Suppose $Z_0\subset Z_1\subset\ldots\subset Z_n\subset\ldots$ is an
infinite sequence of spaces  with $H$-action. Assume the inclusions
$Z_i \into Z_{i+1}$ are $H$-equivariant 
and let
$Z=\bigcup_{i=0}^\infty Z_i$ be their union. 
The \textbf{infinite mapping telescope} of~$Z$ (cf. \cite{Mil62}) is
by definition the space
\begin{equation}
\begin{split}
T_Z & :=Z_0\times[0,1]\cup Z_1\times[1,2]\cup\ldots\cup Z_i\times[i,i+1]\cup
\ldots \\
 & \subseteq Z\times\RR.
\end{split}
\end{equation}
The $H$-action on $Z \times \RR$ given by $g \cdot (z,t) = (gz, t)$
induces a $H$-action on the infinite mapping telescope.

\begin{proposition}
Let $Z_i, Z$, and $T_Z$ be as above. Assume that $Z$ is paracompact.
If for all $x\in Z$ there exists~$i$ such that $x$ lies in the interior
of~$Z_i$, then the natural projection map $\pi_1:T_Z\to Z$ is an $H$-homotopy
equivalence.
\end{proposition}

\begin{proof}
Since the group $H$ is compact, we may 
use an $H$-invariant 
partition of unity to construct a map $f:Z\to[0,\infty)$ such
that $f(x)\ge i+1$ for $x\not\in Z_i$.
Then $g(x):=\bigl(x,f(x)\bigr)$ is an
$H$-equivariant homeomorphism from $Z$ to $g(Z)\subset T_Z$,
and the inclusion $j:g(Z)\rInto T_Z$ is an $H$-equivariant deformation
retraction and satisfies $\pi_1\circ j\circ g=1_Z$.
Therefore $\pi_1$ is an $H$-equivariant homotopy equivalence, as
desired. 
\end{proof}

\begin{theorem}\label{limithomotequiv}
Let $Z=\bigcup_{i=0}^\infty Z_i$ and let $U=\bigcup_{i=0}^\infty U_i$.
Assume that $Z$ and~$U$ are paracompact.
Let $f:Z\to U$ be a continuous $H$-equivariant map such that for each~$n$,
$f(Z_i)\subset U_i$ and the restriction $f_i:=f\big|_{Z_i}:Z_i\to U_i$ is an
$H$-homotopy equivalence.
Then $f$ is an $H$-homotopy equivalence.
\end{theorem}

\begin{proof}
 See \cite[Appendix A]{Mil63}. All the maps in the cited
 reference are equivariant.
\end{proof}

The preceding discussion, together with Theorem~\ref{theorem:milnorlimit}
and the fact that 
$$\bigcup_r S\Gr_{\sandwich,r}(\calK)=\bigcup_r S{\Gr'}_{\sandwich,r}(\calK)
=\sgrzsw$$
yields the following result. 

\begin{theorem}\label{theorem:poly into psm}
 The inclusion $S\Gr_{\sandwich}(\calK)\to S\Gr_{\psm}(\calK)$ is a
$G$-homotopy equivalence.
Equivalently, $\Omega_{\poly}SU(2)\to\Omega_{\psm}SU(2)$ is a $G$-homotopy
equivalence.
\end{theorem}

We also quote the following from \cite{HarSel08}. 

\begin{theorem}\label{theorem:psm-into-smooth}
{\ }

Let \(n \in \Z\) with \(n>0,\).
The natural inclusion $\Omega_{\psm}U(n)\rInto\Omega U(n)$ is an
$SU(n)$-equivariant homotopy
equivalence.
\end{theorem}

\begin{proof} 
The proof is again an application of Theorem \ref{limithomotequiv},
and is given in detail in 
 \cite{HarSel08}.
\end{proof}

Using Theorem~\ref{theorem:module for poly loops} together with 
Theorems~\ref{theorem:poly into psm} and~\ref{theorem:psm-into-smooth})
we can now describe the $R(G)$-module and $R(T)$-module structure
of $K^*_G(\Omega G)$ and~$K^*_T(\Omega G)$.

\begin{theorem}\label{theorem:module structure}
Let $G=SU(2)$ and let $T$ denote its maximal torus. Let $\Omega G$ denote the space of based loops in $G$,
equipped with the pointwise conjugation action of $G$. The
$R(G)$-module (respectively $R(T)$-module) $K^*_G(\Omega G)$
(respectively $K^*_T(\Omega G)$) can be described as follows: 
\begin{align*}
K^q_G(\Omega G)&\cong
K^q_G(\Omega_{\poly} G)
\cong \varprojlim\, K^q_G(\Omega_{\poly,r}G) 
\cong \begin{cases}\prod_{r=0}^{\infty} R(G)&\text{if $q$ is even,}\cr
0&\text{if $q$ is odd};\cr
\end{cases}\cr
K^q_T(\Omega G)&\cong
K^q_T(\Omega_{\poly} G)
\cong \varprojlim\, K^q_T(\Omega_{\poly,r}G) 
\cong \begin{cases}\prod_{r=0}^{\infty} R(T)&\text{if $q$ is even,}\cr
0&\text{if $q$ is odd}.\cr
\end{cases}\cr
\end{align*}
\end{theorem}

\begin{remark} 
Note that our inverse limit becomes a direct product rather
than a direct sum. The result, although a limit of free $R(G)$-modules, is not itself a free $R(G)$-module. (Recall from \cite{Baer37} that 
$\prod_{n=0}^\infty\ZZ$ is not a free abelian group.)
\end{remark}

Finally, we point out that our explicit computation implies in particular that,
in this case, the $W$-invariants of $K^*_G(\Omega G)$ is precisely
$K^*_T(\Omega G)$. (As we noted in the Introduction, this is not true of all $G$-spaces, cf. for
instance \cite[Example 4.8]{HarLanSja09}.) 

\begin{corollary}
$K^*_G(\Omega G) = K^*_T(\Omega G)^W$.
\end{corollary}

\begin{proof}
Since $R(G)=R(T)^W$, this follows immmediately from the right hand
sides of the equalities given in Theorem~\ref{theorem:module structure}. 
\end{proof}

\def\cprime{$'$}


\begin{thebibliography}{10}

\bibitem{AMM} A.~Alekseev, A.~Malkin, E.~Meinrenken,
\newblock Lie group valued moment maps. 
\newblock {\em. J. Differential Geometry} 48:445--495, 1998.


\bibitem{Atiyah} M.~Atiyah, 
\newblock {\em K-Theory}.
\newblock W.A. Benjamin Inc., 1967

\bibitem{Ati68}
M.~Atiyah. 
\newblock Bott periodicity and the index of elliptic operators. 
\newblock {\em Quart. J. Math. Oxford},  Ser. (2), 19:113--140, 1968. 


\bibitem{AtiyahSegal:1965} M.~Atiyah, G.~Segal,
\newblock Equivariant $K$-theory. 
\newblock Notes by R. Schwarzenberger. 
\newblock University of Warwick (mimeographed notes), 1965. 


\bibitem{Ati-Seg}
M.~F. Atiyah and G.~B. Segal.
\newblock Equivariant {$K$}-theory and completion.
\newblock {\em J. Differential Geometry}, 3:1--18, 1969.

\bibitem{Baer37} R.~Baer, Abelian groups without elements of finite order,
{\em Duke Math. J.} {\bf 3}  68--122, 1937.

\bibitem{Borel} A. Borel et al.
\newblock  Seminar on transformation groups. 
\newblock {\em Annals of Mathematics Studies} 46. Princeton, 1960.
\bibitem{BryZha00}
J.-L. Brylinski and B.~Zhang.
\newblock Equivariant {$K$}-theory of compact connected lie groups.
\newblock {\em J. of {$K$}-theory}, 20(1):23--36, 2000.




\bibitem{Greenlees} J.~Greenlees, An introduction to 
equivariant $K$-theory. In J.P.~May,  
\newblock {\em Equivariant homotopy and cohomology theory}, volume~91 of {\em
  CBMS Regional Conference Series in Mathematics}.
\newblock Published for the Conference Board of the Mathematical Sciences,
  Washington, DC, 1996.


\bibitem{HHH05}  M.~Harada, A.~Henriques, T.~Holm, Computation of generalized 
equivaraint cohomologies of Kac-Moody flag varieties, {\em Adv. Math.}
{\bf 197}, 198--221 (2005).

\bibitem{HarJefSel12b} M.~Harada, L.~Jeffrey, P.~Selick, The product
structure on the equivariant $K$-theory of the based loop group of $SU(2)$, in preparation. 

\bibitem{HarLan07}
M.~Harada and G.~D. Landweber.
\newblock Surjectivity for {H}amiltonian {$G$}-spaces in {$K$}-theory.
\newblock {\em Trans. Amer. Math. Soc.}, 359:6001--6025, 2007.

\bibitem{HarLanSja09}
M.~Harada and G.~D.Landweber and R.~Sjamaar. 
\newblock Divided difference operators and character formulae in
equivariant {K}-theory. 
\newblock To be published in {\em Math. Res. Lett.} 


\bibitem{HarSel08}
M.~Harada and P.~Selick.
\newblock Kirwan surjectivity in {$K$}-theory for hamiltonian loop group
  quotients.
\newblock {\em Quart. J. of Math.}
 61 (2010), no. 1, 69--86.

\bibitem{Ill72}
S.~Illman.
\newblock {\em Equivariant Algebraic Topology}, Ph.D. Thesis, Princeton University, 1972.

\bibitem{James1955} I.~James, Reduced product spaces. 
{\em Ann. Math.} {\bf 62} (1955) 170--97.

\bibitem{Kostant-Kumar}
B.~Kostant, S.~Kumar, {\em Kac-Moody Groups, Their Flag Varieties and
Representation Theory}, Birkh\"auser (Progress in Mathematics vol. 204)
2002. 
\bibitem{KK90} B.~Kostant, S.~Kumar,$T$-equivariant $K$-theory of 
generalized flag varieties, {\em J. Diff. Geom.}
{\bf 32} (1990) 549-603.

\bibitem{May96}
J.~P. May.
\newblock {\em Equivariant homotopy and cohomology theory}, volume~91 of {\em
  CBMS Regional Conference Series in Mathematics}.
\newblock Published for the Conference Board of the Mathematical Sciences,
  Washington, DC, 1996.

\bibitem{Mil62}
J.~Milnor.
\newblock On axiomatic homology theory.
\newblock {\em Pacific J. Math.}, 12:337--341, 1962.

\bibitem{Mil63}
J.~Milnor.
\newblock {\em Morse theory}.
\newblock Based on lecture notes by M. Spivak and R. Wells. Annals of
  Mathematics Studies, No. 51. Princeton University Press, Princeton, N.J.,
  1963.

\bibitem{Mil74}
J.~Milnor and J.~Stasheff.
\newblock {\em Characteristic Classes}.
\newblock Annals of
  Mathematics Studies, No. 76. Princeton University Press, Princeton, N.J.,
  1974.

\bibitem{Mit86}
S.~Mitchell.
\newblock A filtration of the loops on {$SU(n)$} by Schubert varieties.
\newblock {\em Math. Z.}, 193(3):347--362, 1986.

\bibitem{PS86}
A.~Pressley and G.~Segal.
\newblock {\em Loop groups}.
\newblock Oxford Mathematical Monographs. The Clarendon Press Oxford University
  Press, New York, 1986.

\bibitem{Seg68}
G.~Segal.
\newblock Equivariant {$K$}-theory.
\newblock {\em Inst. Hautes \'Etudes Sci. Publ. Math.}, 34:129--151, 1968.

\bibitem{Sel97}
P.~Selick.
\newblock {\em Introduction to homotopy theory}, volume~9 of {\em Fields
  Institute Monographs}.
\newblock American Mathematical Society, Providence, RI, 1997.

\end{thebibliography}
\end{document}